\documentclass[aop]{imsart}

\RequirePackage{amsthm,amsmath,amsfonts,amssymb}
\RequirePackage[numbers]{natbib}

\usepackage{amsmath, amsthm, amsfonts, amssymb, amsbsy, bigstrut, graphicx, enumerate,  upref, longtable, comment, booktabs, array, caption, subcaption}

\usepackage[breaklinks]{hyperref}

\startlocaldefs
\numberwithin{equation}{section}
\theoremstyle{plain}
\newtheorem{thm}{Theorem}[section]
\newtheorem{lmm}[thm]{Lemma}
\newtheorem{cor}[thm]{Corollary}

\theoremstyle{remark}

\newtheorem{remark}[thm]{Remark}

\newcommand{\tf}{\tilde{f}}

\newcommand{\ee}{\mathbb{E}}

\newcommand{\mf}{\mathcal{F}}

\newcommand{\rr}{\mathbb{R}}

\newcommand{\zz}{\mathbb{Z}}

\numberwithin{equation}{section}

\newcommand{\Z}{\mathbb{Z}}

\newcommand{\R}{\mathbb{R}}

\newcommand{\E}{\mathbb{E}}

\renewcommand{\P}{\mathbb{P}}

\renewcommand{\tilde}{\widetilde}

\newcommand{\poly}{\phi_{\textup{poly}}}

\newcommand{\fpoly}{f^{\textup{poly}}}

\endlocaldefs

\begin{document}

\begin{frontmatter}
\title{An invariance principle for the 1D KPZ equation}
\runtitle{Invariance principle for 1D KPZ}

\begin{aug}
\author[A]{\fnms{Arka}~\snm{Adhikari}\ead[label=e1]{arkaa@stanford.edu}}
\and
\author[B]{\fnms{Sourav}~\snm{Chatterjee}\ead[label=e2]{souravc@stanford.edu}}
\address[A]{Department of Mathematics, Stanford University, Stanford, CA\printead[presep={,\ }]{e1}}

\address[B]{Departments of Mathematics and Statistics, Stanford University, Stanford, CA\printead[presep={,\ }]{e2}}
\end{aug}

\begin{abstract}
Consider a discrete one-dimensional random surface whose height at a point grows as a function of the heights at neighboring points plus an independent random noise. Assuming that this function is equivariant under constant shifts, symmetric in its arguments, and at least six times continuously differentiable in a neighborhood of the origin, we show that as the variance of the noise goes to zero, any such process converges to the Cole--Hopf solution of the 1D KPZ equation under a suitable scaling of space and time. This proves an invariance principle for the 1D KPZ equation, in the spirit of Donsker's invariance principle for Brownian motion.
\end{abstract}

\begin{keyword}[class=MSC]
\kwd[Primary ]{60F05}
\kwd{82C05}
\end{keyword}

\begin{keyword}
\kwd{KPZ equation}
\kwd{scaling limit}
\kwd{invariance principle}
\kwd{KPZ universality}
\end{keyword}

\end{frontmatter}


\section{Introduction}\label{sec:result}
\subsection{Main result}\label{sec:main}
Let $\Z$ be the set of integers and $\Z_+$ be the set of nonnegative integers. Let $(y(x,t))_{x\in \Z, t\in \Z_+}$ be a collection of i.i.d.~random variables with mean zero and finite moment generating function in a neighborhood of zero. We will collectively refer to these variables as the `noise variables'. Let $\psi:\R^2 \to \R$ be any function. Consider a growing random surface $f_N:\Z_+\times \Z \to \R$ defined through the recursion
\begin{align}\label{eq:fndef}
f_N(x, t) = \psi(f_N(x-1,t-1), f_N(x+1,t-1)) + N^{-1/4}y(x, t), 
\end{align}
with $f_N(x,0) = 0$ for all $x$, where $N$ is an integer that we will eventually send to infinity. Here $f_N(x,t)$ denotes the height of the surface at location $x$ at time $t$. The above equation means that this height is a function of the heights at location $x-1$ and $x+1$ at time $t-1$, plus a random noise. We assume that $\psi$ has the following properties.
\begin{itemize}
    \item {\it Equivariance under constant shifts.} For any $u,v,c\in \R$, $\psi(u+c,v+c) = \psi(u,v)+c$.
    \item {\it Symmetry.} For all $u,v\in \R$, $\psi(u,v) = \psi(v,u)$.
    \item {\it Regularity.} $\psi$ is at least six times continuously differentiable in a neighborhood of the origin. 
\end{itemize}
The above assumptions are natural from a physical point of view. An example of a $\psi$ satisfying the above assumptions is 
\begin{align}\label{eq:disc1}
\psi(u,v) = \frac{u+v}{2}+(u-v)^2,
\end{align}
which represents a `discrete version' of KPZ growth. Another example, considered in the original paper of Kardar, Parisi and Zhang~\cite{kardaretal86}, is 
\begin{align}\label{eq:disc2}
\psi(u,v) = \frac{u+v}{2} + \sqrt{1+(u-v)^2}.
\end{align}
Our main result is that under the above assumptions, if we subtract off a certain deterministic multiple of $t$ from $f_N(x,t)$, the resulting process converges in law to the Cole--Hopf solution of the 1D KPZ equation as $N\to \infty$ under parabolic scaling of space and time (see Subsection \ref{sec:background} for background on KPZ). We now give a careful statement of this result.

Let $\R$ be the set of real numbers and $\R_+$ be the set of nonnegative real numbers. We will now define a process $\tf_N :\R\times \R_+\to \R$ by rescaling space and time in the definition of $f_N$ and subtracting off a deterministic linear drift term. First, we need to define two constants. The first constant is 
\[
\beta := \partial_1^2 \psi(0,0),
\]
where $\partial_1^2$ denotes the second partial derivative in the first coordinate. For $x\in \Z$ and $t\in \Z_+$, let $p(x,t)$ denote the probability that a simple symmetric random walk on $\Z$, started at $0$ at time $0$, is at $x$ at time $t$. Let 
\[
\Delta(x,t) := p(x+1,t) - p(x-1,t).
\]
Next, define 
\begin{align}\label{eq:cdef}
c := \frac{1}{24} \partial_1^4 \psi(0,0) + \frac{\beta^3}{12}.
\end{align}
Finally, we define the second constant
\begin{align}\label{eq:vdef}
V := c\biggl[ \sum_{x\in \Z}\sum_{t=0}^\infty \Delta(x,t)^4 (\mu_4 - \mu_2^2)  + \biggl(\sum_{x\in \Z}\sum_{t=0}^\infty  \Delta(x,t)^2 \mu_2 \biggr)^2\biggr],
\end{align}
where $\mu_k$ denotes the $k^{\mathrm{th}}$ moment of the noise variables. 

Having defined $\beta$ and $V$, we now define the rescaled and renormalized surface growth process $\tf_N:\R\times\R_+\to \R$. For  any $(x,t)\in \R \times \R_+$ such that $x$ is an integer multiple of $N^{-1/2}$ and $t$ is an integer multiple of $N^{-1}$, let 
\begin{align}
\tf_N(x,t) &:= f_N(\sqrt{N} x, N t) \notag \\
&\qquad - \biggl(V + \frac{1}{2}\beta N^{1/2} \mu_2 + \frac{1}{6}\beta^2 N^{1/4} \mu_3 + \frac{1}{24}\beta^3 (\mu_4-3\mu_2^2) + N \psi(0,0)\biggr)t.\label{eq:tfndef0}
\end{align}
Note that the renormalization term depends only on the first four moments of the noise variables. 
For all other $(x,t)$, define $\tf_N(x,t)$ by linear interpolation. (The exact method of linear interpolation will be described in Section \ref{sec:mainproof}.) 
The following is the main result of this paper. 
\begin{thm}\label{thm:mainresult}
Let $\tf_N$ be defined as above, and suppose that $\beta \ne 0$. Then the $C(\R\times \R_+)$-valued random function $\exp(\beta \tf_N)$
converges in law as $N \to \infty$ to a solution $\mathcal{Z}$ of the stochastic heat equation with multiplicative noise
\begin{equation}\label{eq:she}
\partial_t  \mathcal{Z}  = \frac{1}{2} \partial_x^2  \mathcal{Z} + \sqrt{2\mu_2} \beta \mathcal{Z}\xi, \ \ \   \mathcal{Z}(0,\cdot) \equiv 1, 
\end{equation}
where $\xi$ is standard space-time white noise. Since $\beta^{-1} \log \mathcal{Z}$ is the Cole--Hopf solution of the KPZ equation displayed in equation \eqref{eq:kpzequation} below, this means that $\tf_N$ converges in law to the Cole--Hopf solution of the KPZ equation. The topology on $C(\R\times \R_+)$ that we use here is the topology of uniform convergence on compact sets. If $\beta = 0$, then $\tf_N$ converges in law to a solution $h$ of the stochastic heat equation with additive noise 
\begin{equation}\label{eq:she2}
    \partial_t h= \frac{1}{2} \partial_x^2 h + \sqrt{2\mu_2} \xi, \ \ \ h(0,\cdot) \equiv 0. 
\end{equation}

\end{thm}

\begin{remark}
The white noise that appears in equation \eqref{eq:she} has the same properties as the white noise appearing in \cite[Theorem 2.1]{albertsetal14b}. See \cite[Section 3.2]{albertsetal14b} for the precise definition.
\end{remark}
\begin{remark} The result and method could potentially be generalized to other initial types of initial data. Our result is perturbative, so if it is known that some explicitly solvable model will converge to the KPZ equation with some non-zero intitial data, we could generalize the type of interaction to find another interaction that will still converge to the KPZ equation.
\end{remark}

\begin{remark}
The solution of the stochastic heat equation with multiplicative noise \eqref{eq:she} can be explicitly written down using a chaos expansion, as in \cite[Theorem 2.7]{albertsetal14b}. The solution of the stochastic heat equation with additive noise \eqref{eq:she2} is easy to write down explicitly as a convolution of the heat kernel and white noise. 
\end{remark}
\begin{remark}
Just for fun, let us compute the constants appearing in the renormalization term when $\psi$ is given by \eqref{eq:disc1} and \eqref{eq:disc2}. First, suppose that $\psi$ is the function displayed in \eqref{eq:disc1}. Then $\partial_1^2 \psi\equiv 2$, and so $\beta=2$. Also, $\partial_1^4 \psi \equiv 0$, and hence $c= \beta^3/12 = 2/3$. Thus, 
\[
V = \frac{2}{3}(C_1 \mu_4 + C_2 \mu_2^2),
\]
where $C_1$ and $C_2$ are absolute constants given by the formulas 
\[
C_1 = \sum_{x\in \Z}\sum_{t=0}^\infty \Delta(x,t)^4, \ \ \  C_2 = \biggl(\sum_{x\in \Z} \sum_{t=0}^\infty \Delta(x,t)^2\biggr)^2 - C_1.
\]
(It is not clear if $C_1$ and $C_2$ can be evaluated in closed form. The numerical values, calculated by summing up to $t=10000$, are roughly $C_1 \approx 2.16$ and $C_2\approx 13.70$.) Next, consider the $\psi$ displayed in equation~\eqref{eq:disc2}. For this $\psi$, note that
\[
\partial_1^2 \psi(u,v) = \frac{1}{\sqrt{1+(u-v)^2}} - \frac{(u-v)^2}{(1+(u-v)^2)^{3/2}},
\]
which gives $\beta = \partial_1^2\psi(0,0) = 1$. Differentiating further, we get
\begin{align*}
\partial_1^4 \psi(u,v) = -\frac{3}{(1+(u-v)^2)^{3/2}} + \text{ some multiple of $u-v$},
\end{align*}
which gives $\partial_1^4 \psi(0,0) = -3$, and hence $c = -1/24$. Thus, $V = -(C_1\mu_4+C_2\mu_2^2)/24$. 
\end{remark}



\subsection{Background}\label{sec:background}
Let $f(x,t)$ denote the height of an evolving one-dimensional random interface at time $t\in \R_+$ and location $x\in \R$, where $\R$ is the set of real numbers and $\R_+$ is the set of nonnegative real numbers. The evolution of the interface is said to follow 1D Kardar--Parisi--Zhang (KPZ) equation~\cite{kardaretal86} if it formally satisfies the stochastic partial differential equation
\[
\partial_t f = \alpha \partial_x^2 f + \beta (\partial_x f)^2 + \gamma \xi,
\]
where $\xi$ is space-time white noise, and $\alpha$, $\beta$ and $\gamma$ are real-valued parameters. If $\alpha$ and $\gamma$ are nonzero, then by suitably scaling space and time, we can convert the above equation to the equation 
\begin{align}\label{eq:kpzequation}
\partial_t f = \frac{1}{2}\partial_x^2 f + \frac{\beta }{2}(\partial_x f)^2 + \sqrt{2}\xi.
\end{align}
which has only a single parameter, $\beta$. 
One can give a rigorous meaning to \eqref{eq:kpzequation} by declaring that a solution $f$ can be obtained as $f = \beta^{-1}\log\mathcal{Z}$, where $\mathcal{Z}$ is a solution to the stochastic heat equation with multiplicative noise~\eqref{eq:she} with $\mu_2=1$, which is a rigorously defined SPDE~\cite{walsh86, mueller91}. Indeed, a formal calculation using It\^{o} calculus shows that if $\mathcal{Z}$ solves \eqref{eq:she} with $\mu_2=1$, then $f$ must solve \eqref{eq:kpzequation}. This is known as the `Cole--Hopf solution' of the 1D KPZ equation, first proposed in \cite{bertinigiacomin97}. It is known to be equivalent to the pathwise solutions constructed later using the theory of regularity structures~\cite{hairer13, hairer14}.

The `weak universality conjecture' for the 1D KPZ equation says that any 1D interface growth process that is driven by microscopic fluctuations, where the heights at neighboring points have a nontrivial effect on the growth of the height at a point, should converge to the KPZ equation in some suitable scaling limit~\cite{bertinietal94}. This is admittedly rather ill-posed, but it is one of the things that make the KPZ equation an object of central  interest. 

There is now considerable evidence in favor of the weak universality conjecture, mainly in the form of rigorously proved convergences of various discrete growth models to the KPZ equation in a suitable space-time scaling limit. Some examples are:
\begin{itemize}
\item Asymmetric exclusion processes in the weakly asymmetric limit~\cite{bertinigiacomin97, amiretal11, dembotsai16, yang20, yang20c, yang21}.
\item Models belonging to the class of stochastic vertex models~\cite{borodincorwin14, corwintsai17,corwinghosaletal18,yierlin19}.
\item Directed random polymers in the intermediate disorder regime~\cite{albertsetal10, albertsetal14a, albertsetal14b, morenofloresetal13}.
\item A large class of stationary, weakly asymmetric, conservative particle systems~\cite{goncalvesjara14}. The limit here is the energy solution of the KPZ equation, which was later shown to be unique in \cite{gubinelliperkowski17}. A more general result along the same line was proved later in \cite{diehletal17}.
\item KPZ equation with smoothed nonlinearity, taken to the limit where smoothing is removed~\cite{funakiquastel15}.
\item The KPZ equation with $(\partial_x f)^2$ replaced by $F(\partial_x f)$ for some general nonlinear function $F$, under appropriate limits of scaling space and time~\cite{hairerquastel18, hairerxu19}, confirming conjectures from \cite{halpinhealyzhang95, krugspohn91}.
\end{itemize}
A more complete list of references with a more extensive discussion can be found in~\cite[Section 6]{corwinshen20}. Of the papers cited above, our result is perhaps most closely related to the results of \cite{hairerquastel18, hairerxu19}, which therefore deserve some elaboration. In \cite{hairerquastel18}, the following two classes of SPDEs were considered:
\begin{align}\label{eq:hq1}
\partial_t f_\epsilon = \partial_x^2 f_\epsilon + F(\partial_x f_\epsilon) + \sqrt{\epsilon} \xi,
\end{align}
and 
\begin{align}\label{eq:hq2}
\partial_t f_\epsilon = \partial_x^2 f_\epsilon + \sqrt{\epsilon} F(\partial_x f_\epsilon) + \xi,
\end{align}
where $\xi$ is white noise, $\epsilon$ is a parameter, and $F$ is an even function which was taken to be a polynomial in \cite{hairerquastel18} and extended to a larger class of functions in \cite{hairerxu19}. If $f_\epsilon$ is a solution to \eqref{eq:hq1}, then it was shown in \cite[Theorem 1.1]{hairerquastel18} that the rescaled map $(x,t)\mapsto f_\epsilon(\epsilon^{-1} x, \epsilon^{-2}t)$ converges in law to a solution of the KPZ equation as $\epsilon\to 0$, after subtracting off suitable renormalization terms. On the other hand, if $f_\epsilon$ is a solution to \eqref{eq:hq2}, then \cite[Theorem 1.2]{hairerquastel18} shows that the rescaled map $(x,t)\mapsto \sqrt{\epsilon}f_\epsilon(\epsilon^{-1} x, \epsilon^{-2}t)$ converges in law to a solution of the KPZ equation as $\epsilon\to 0$, after subtracting off suitable renormalization terms. These results hold when $F''(0)\ne 0$. If $F''(0)=0$, then a different scaling is needed, depending on the smallest $k$ such that the $k^{\textup{th}}$ derivative of $F$ at $0$ is nonzero.

The limit corresponding to \eqref{eq:hq1} is known as the `intermediate disorder scaling limit', while the one corresponding to \eqref{eq:hq2} is known as the `weakly asymmetric scaling limit'. In a sense, our Theorem \ref{thm:mainresult} can be viewed as a version of weak universality for the 1D KPZ equation in the intermediate disorder regime.  The reason is that under the equivariance and symmetry assumptions, $\psi$ can be expressed as
\[
\psi(u,v) = \frac{u+v}{2} + \phi(u-v)
\]
for some even function $\phi$ (see details in Section \ref{sec:prelim}), which implies that the recursion \eqref{eq:fndef} can be seen as a discretized version of \eqref{eq:hq1}. The major difference between the framework of \cite{hairerquastel18, hairerxu19} and ours is that we start from discrete growth processes rather than solutions of SPDEs. Two other differences are that we consider general i.i.d.~noise instead of Gaussian noise, and our argument does not make use of heavy machinery like regularity structures. 

From a slightly different perspective, Theorem \ref{thm:mainresult} is an invariance principle for the 1D KPZ equation, analogous to Donsker's invariance principle for Brownian motion~\cite{donsker51}. Indeed,  consider a $(0+1)$-dimensional process $f_N(t)$ growing according to the recursion 
\[
f_N(t) = \psi(f_N(t-1)) + N^{-1/2} y(t),
\]
where now $(y(t))_{t\in \Z_+}$ are i.i.d.~mean zero random variables and $\psi$ is a function from $\R$ into $\R$. Suppose that $\psi$ is equivariant under constant shifts. Since $\psi$ is now a function of only one variable, equivariance under constant shifts implies that $\psi$ must be the form $\psi(u) = u + c$ for some $c\in \R$. Thus, if $f_N(0)=0$, then
\[
f_N(t) = ct + N^{-1/2}\sum_{s=1}^t y(s). 
\]
So, if we rescale and renormalize $f_N$ as
\[
\tf_N(t) := f_N(Nt) - cNt = N^{-1/2}\sum_{s=1}^{Nt} y(s),
\]
then by Donsker's theorem, $\tf_N$ converges in law to Brownian motion. This way of writing Donsker's theorem makes it clear why Theorem \ref{thm:mainresult} is a `KPZ version' of it. While Donsker's theorem shows invariance of the scaling limit under various choices of the law of the noise variables, our result proves invariance both under changing the law of the noise variables and the choice of $\psi$ (subject to the constraints of equivariance under constant shifts, symmetry and regularity).

A natural problem, then, is to investigate whether the assumptions of Theorem \ref{thm:mainresult} can be relaxed. Donsker's theorem requires the noise variables to have only finite second moments, whereas Theorem \ref{thm:mainresult} needs finite moment generating function in a neighborhood of zero. Can this be relaxed? Also, in \eqref{eq:fndef}, can $f_N(x,t)$ be a function of $f_N(x\pm i,t-1)$ for $i=0, 1,\ldots,k$ for some fixed $k$ plus noise, instead of just $f(x\pm 1, t-1)$ as we currently have? Depending on the type of interaction, such as a higher order interaction, it may be sufficient to perturb from the directed polymer itself. Generalization beyond i.i.d.~noise is also an interesting question. Does it suffice to have the noise `homogenize' in a certain sense, as it happens for exclusion processes? 

A different kind of invariance principle, similar to Donsker's theorem in that the invariance is only in the law of the noise, was obtained in~\cite{hairershen17}. Roughly speaking, the main result of \cite{hairershen17} is that a solution of the KPZ equation with non-Gaussian noise and coefficients depending in a certain way on a parameter $\epsilon$, converges to a solution of the usual KPZ equation (with white noise) as $\epsilon \to 0$. 

The idea of looking at discrete growth processes growing according to \eqref{eq:fndef}, with $\psi$ satisfying the equivariance, symmetry and other conditions, was introduced in series of papers~\cite{chatterjee21,chatterjee21c,chatterjeesouganidis21}. In \cite{chatterjee21}, it was shown that in the absence of noise, any such process converges to a solution of the deterministic KPZ equation under parabolic scaling. This result was extended to a larger class of deterministically growing processes, with novel scaling limits, in \cite{chatterjeesouganidis21}. A kind of `local KPZ universality' result for such processes was established in \cite{chatterjee21c}.

In the context of the above discussion, it should be noted that weak KPZ universality is fundamentally different than `strong KPZ universality', which says that the `long time' scaling limit of a large class of growth processes, which are typically driven by non-vanishing noise, is a Markov process known as the `KPZ fixed point'~\cite{matetskietal21, quastelsarkar22}. The strong universality conjecture is at present well out of the reach of available techniques. Its predictions have been verified only in integrable models where exact calculations are possible (see \cite{ganguly22, spohn20, remenik22} for surveys).

Incidentally, Theorem \ref{thm:mainresult} has implications about convergence of discretized versions of the KPZ equation to the continuum limit. For example, a natural discretization would be to take something like \eqref{eq:disc1}. 
Theorem \ref{thm:mainresult} shows that such discretizations converge to the Cole--Hopf solution of the KPZ equation after subtracting off the correct renormalization term, which is explicitly given by the formulas from Subsection \ref{sec:main}. A number of papers have been written recently about convergence of discretized stochastic PDEs to their continuum limits. For example, in~\cite{cannizzaromatetski18} it is shown that a discretization of the stochastic Burgers equation (which is formally the equation for the derivative of a solution of the KPZ equation) converges to the correct continuum limit. General theories of convergence of discretized SPDEs using regularity structures have been developed in~\cite{hairermatetski18, erhardhairer19} (see~\cite{frizhairer20} for background). The key difference between these works and ours is that we have a general growth mechanism encoded by the function $\psi$ and general i.i.d.~noise, with a focus towards KPZ universality, whereas a typical paper on convergence of discretized SPDEs (e.g., the ones cited above) would consider a specific discretization rule and Gaussian noise.

\subsection{Outline of the paper} After we introduce our conventions and models in Sections \ref{sec:prelim} and \ref{sec:polymer}, we begin the proof of Theorem \ref{thm:mainresult}. Our method is an inductive approach detailed in Section \ref{sec:mainarg}. In that section, we propose the following ansatz on the form of the partition function $f_N$.
\begin{equation*}
    f_N = \fpoly_N + Y_N + \delta_N,
\end{equation*}
$\fpoly_N$ is the log-partition function for the directed polymer model at inverse temperature $\beta N^{-1/4}$, $Y_N$ is a nonlinear function of the polymer process that acts as a renormalization term, and $\delta_N$ is an error term which we try to show is $o(1)$ as $N\to\infty$. The main result of section \ref{sec:mainarg} shows that the error propagation term $\delta_N$ indeed tends to zero; this is Theorem \ref{thm:errorprop2}.

There are many preliminary estimates that are necessary in order to prove Theorem \ref{thm:errorprop2}; these  estimates involve understanding the differences $\fpoly_N(x+1,t) - \fpoly_N(x-1,t)$ of the polymer log-partition function, as well as properties of the renormalization term $Y_N$.  Most of these estimates are obtained in Sections \ref{sec:boundspoly} through \ref{sec:furtherestim}. Section \ref{sec:concrenom} shows that  the renormalization term $Y_N$ is actually close to a linear function of $t$ with high probability, and Section \ref{sec:mainproof} derives Theorem \ref{thm:mainresult} from Theorem \ref{thm:errorprop2} using the known result about convergence of the polymer model to KPZ at intermediate disorder~\cite{albertsetal14b}. Finally, Section \ref{sec:betazero} briefly treats the $\beta=0$ case; the analysis is very similar to the $\beta\ne 0$ case but much simpler.



\section{Conventions}\label{sec:prelim}
Throughout the rest of the manuscript, we will adopt the convention that $A\lesssim B$ means that $A\le CB$ for some deterministic positive real number $C$ that does not depend on $N$, $x$ or $t$, as long $(x,t)$ is in some given rectangle of the form $[-aN, aN] \times [0, bN]$. Here $N$ is the parameter from Section \ref{sec:result} that we will eventually send to infinity, and $x$ and $t$ are specific choices of space and time points where we want to prove something. We will write $A = O(B)$ if $|A|\lesssim |B|$, and $A = o(B)$ if $A/B\to 0$ uniformly over $(x,t)$ in $[-aN, aN] \times [0, bN]$ as $N\to\infty$. We will use the notation $A\ll B$ to mean that $A\le CB$ for some sufficiently small positive constant $C$, where `sufficiently small' means `as small as we need, but not depending on $N$, $x$ or $t$'. 

We will often write sentences like ``there is an event $\Omega$ with $\P(\Omega)=1-o(1)$ such that on $\Omega$, we have that for all $(x,t)\in [-aN,aN]\times [0,bN]$, $|G(x,t)|\lesssim N^{-\alpha}$'', where $G$ is a random  function and $\alpha$ is a constant. What this will mean is that there is an event $\Omega$ that may potentially vary with $N$, with $\P(\Omega)\to 1$ as $N\to \infty$, and  there is some deterministic constant $C$, independent of $N$, such that on $\Omega$, we have
\[
\max_{(x,t)\in [-aN,aN]\times[0,bN]} |G(x,t)|\le CN^{-\alpha}.
\]
Next, we make some reductions and simplifications to our growth process. First, if $\beta\ne 0$, then instead of \eqref{eq:tfndef0}, we will define $\tf_N$ as 
\begin{align}\label{eq:tfndef}
\tf_N(x,t) := f_N(\sqrt{N} x, N t) - \biggl(V + \frac{N}{\beta}\log m(N^{-1/4} \beta) + N \psi(0,0)\biggr)t,
\end{align}
where $m$ denotes the moment generating function of the noise variables. 
We claim that it suffices to prove Theorem \ref{thm:mainresult} with this new definition of $\tf_N$. Indeed, note that by the cumulant expansion,
\begin{align*}
\frac{N}{\beta}\log m(N^{-1/4} \beta) &= \frac{1}{2!}\beta N^{1/2} \mu_2 + \frac{1}{3!}\beta^2 N^{1/4} \mu_3 + \frac{1}{4!}\beta^3 (\mu_4-3\mu_2^2) + O(N^{-1/4}).
\end{align*}
This shows that the difference between the old $\tf_N$ and the new $\tf_N$ converges to zero uniformly on compact sets as $N\to\infty$. 

We will assume throughout that $\psi(0,0)=0$. There is no loss of generality in this, because of the following. Suppose we let $\psi_0(u,v) := \psi(u,v)-\psi(0,0)$, and define $g_N$ using $\psi_0$ just as we defined $f_N$ using $\psi$. Then it is easy to prove by induction, using the equivariance property of $\psi$, that for all $x$ and $t$,
\[
g_N(x,t) = f_N(x,t) - \psi(0,0)t.
\]
From this, it is easy to see that if Theorem \ref{thm:mainresult} holds for $g_N$, then it also holds for $f_N$. 

Fixing $N$, we will denote by $f$ the function 
\begin{align}\label{eq:fxtdef}
f(x,t) := f_N(x,t) - \frac{t}{\beta} \log m(N^{-1/4}\beta)
\end{align}
defined on $\Z\times \Z_+$, where $f_N$ is the function defined in equation \eqref{eq:fndef}. Under the assumption that $\psi(0,0)=0$, the equivariance property of $\psi$ ensures that $f$ satisfies the recursion
\begin{align}\label{eq:frec}
f(x,t) = \psi(f(x-1,t-1), f(x+1, t-1)) + N^{-1/4} y(x,t) - \frac{1}{\beta} \log m(N^{-1/4}\beta).
\end{align}
Define the function $\phi:\R \to \R$ as
\begin{equation}\label{eq:phiprop}
\phi(u) := \psi\biggl(\frac{u}{2}, -\frac{u}{2}\biggr) = \psi(u,0) - \frac{u}{2},
\end{equation}
where the second equality holds by the equivariance property of $\psi$. 
Also by the equivariance property, note that for any $u,v\in \R$, 
\begin{align*}
\psi(u,v) &= \psi\biggl(u - \frac{1}{2}(u+v), v - \frac{1}{2}(u+v)\biggr) + \frac{u+v}{2}\\
&= \psi\biggl(\frac{u-v}{2}, \frac{v-u}{2}\bigg) + \frac{u+v}{2}\\
&= \phi(u-v) + \frac{u+v}{2}.
\end{align*}
Thus, the recursion \eqref{eq:frec} can be rewritten as
\begin{align}\label{eq:frec2}
f(x,t) &= \frac{1}{2}(f(x-1,t-1)+ f(x+1,t-1)) \notag \\
&\qquad + \phi(f(x+1,t-1)-f(x-1,t-1)) \notag \\
&\qquad + N^{-1/4}y(x,t) - \frac{1}{\beta} \log m(N^{-1/4}\beta).
\end{align}
For future reference, we note that the function $\phi$ defined above is even and is $C^6$ in a neighborhood of zero, by the symmetry and regularity properties of $\psi$. Moreover, we can do the following calculations. Let $\phi^{(k)}$ denote the $k^{\mathrm{th}}$ derivative of $\phi$, and let $\partial_i$ denote differentiation in the $i^{\mathrm{th}}$ coordinate. Then by equation \eqref{eq:phiprop} and the evenness of $\phi$, we have
\begin{equation}\label{eq:phiderivs}
\begin{split}
&\phi(0) = \psi(0,0)=0,\\
&\phi^{(1)}(0) = \phi^{(3)}(0) = \phi^{(5)}(0)=0,\\
 &\phi^{(2)}(0) = \partial_1^2\psi(0,0) = \frac{\beta}{4}, \\ 
&\phi^{(4)}(0) = \partial_1^4\psi(0,0) = -\frac{\beta^3}{8} + 24 c.
\end{split}
\end{equation}
In the rest of manuscript we will work under the assumption that $\beta\ne0$, except in Section \ref{sec:betazero}, where the $\beta=0$ case will be handled.

\section{The directed polymer model}\label{sec:polymer}
In this section we consider a special kind of growing random surface, defined by the model of directed polymers in a random environment (see \cite{comets17, bateschatterjee20} for background on directed polymers). Let $(y(x,t))_{x\in \Z, t\in \Z_+}$ be our noise variables, as before. Recall that $m$ denotes the moment generating function of the noise variables, which is finite in a neighborhood of zero. We define the growing random surface $\fpoly: \Z\times \Z_+ \to \R$ as follows. Let $RW(x,t)$ denote the set of all simple symmetric random walk paths on $\Z$ that that terminate at $x$ at time $t$, regardless of the initial point at time $0$. Let $\fpoly(x,t)=0$ if $t=0$, and for $t>0$, let
\[
\fpoly(x,t) := \frac{1}{\beta} \log \biggl( \frac{1}{2^t} \sum_{S\in RW(x,t)} \prod_{s=1}^t \frac{\exp(\beta N^{-1/4} y(S(s),s))}{m(\beta N^{-1/4})} \biggr).
\]
It is easy to see that $\fpoly$ satisfies the recursion
\begin{align*}
\fpoly(x,t) &= \frac{1}{\beta}\log \biggl(\frac{e^{\beta \fpoly(x-1,t-1)} + e^{\beta \fpoly(x+1,t-1)}}{2}\biggr) + N^{-1/4}y(x,t) - \frac{1}{\beta}\log m(\beta N^{-1/4}). 
\end{align*}
This model will be of fundamental importance in the sequel. Of particular importance is the polymer partition function $X(x,t) := \exp(\beta \fpoly(x,t))$. Note that $X(x,0)=1$ for all $x$.

Note that the above recursion for $\fpoly$ can be rephrased as 
\begin{align}
\fpoly(x,t) &= \frac{1}{2}(\fpoly(x-1,t-1) + \fpoly(x+1,t-1)) \notag \\
&\qquad + \poly(\fpoly(x+1,t-1) - \fpoly(x-1,t-1)) \notag \\
&\qquad + N^{-1/4}y(x,t) - \frac{1}{\beta}\log m(\beta N^{-1/4}),\label{eq:polyrec}
\end{align}
where 
\[
\poly(u) := \frac{1}{\beta}\log\biggl(\frac{e^{\frac{\beta}{2} u} + e^{-\frac{\beta}{2} u}}{2}\biggr). 
\]
Note that this is exactly like the recursion \eqref{eq:frec2}, with $\phi$ replaced by $\poly$. Our specific choice of $\beta$ is motivated by the fact that $\poly$ matches $\phi$ up to the second derivative at zero: Simple calculations show that, in analogy with \eqref{eq:phiderivs}, we have
\begin{equation}\label{eq:polyderivs}
\begin{split}
&\poly(0) = \poly^{(1)}(0) = \poly^{(3)}(0) = \poly^{(5)}(0)=0,\\
&\poly^{(2)}(0) = \frac{\beta}{4},\ \  \poly^{(4)}(0) = -\frac{\beta^3}{8}.
\end{split}
\end{equation}
Note that the fourth derivatives of $\phi$ and $\poly$ do not match. This leads to some substantial complexities later.

\section{Bounds on the polymer partition function} \label{sec:boundspoly}
Recall the random fields $\fpoly$ and $X$ defined in the previous section. In this section we derive some preliminary estimates for these fields. We start with the following lemma which relates differences of $\fpoly$ to those of $X$. 
\begin{lmm} \label{lem:fracdiff}
Take any $a>0$, $b>0$ and $\delta >0$. Suppose that for some realization of the noise variables, 
\begin{equation} \label{eq:fracdiff}
\biggl| \frac{X(x+1,t) - X(x-1,t)}{X(x+1,t) + X(x-1,t)} \biggr| <\delta
\end{equation}
for all $(x,t)\in [-aN, aN]\times [0,bN]$. Then we have that
\begin{equation}\label{eq:controldiff}
|\fpoly(x+1,t) - \fpoly(x-1,t)| \lesssim \biggl|\frac{X(x+1,t) - X(x-1,t)}{X(x+1,t)  +X(x+1,t)}\biggr|,
\end{equation}
where, as mentioned in Section \ref{sec:prelim}, the implicit constant above has no dependence on $N$, $x$ or $t$ (but may depend on $a$, $b$, $\delta$ and $\beta$). Furthermore, if $\delta$ varies with $N$ (with $a$ and $b$ remaining fixed), then we have 
\begin{equation} \label{eq:carefulTaylor}
    |\fpoly(x+1,t) - \fpoly(x-1,t)| = \frac{2}{\beta} \biggl|\frac{X(x+1,t) - X(x-1,t)}{X(x+1,t)  +X(x+1,t)}\biggr|(1+ O(\delta)). 
\end{equation}
\end{lmm}

(We remark here that the fraction on the left in \eqref{eq:fracdiff} is always less than or equal to $1$. The improvement lies in trying to show that it is strictly less than $1$ by a constant.)

\begin{proof}
First, note that
\begin{equation*}
\begin{aligned}
&\fpoly(x+1,t) - \fpoly(x-1,t)\\
& = \frac{1}{\beta}\log X(x+1,t) - \frac{1}{\beta} \log X(x-1,t) \\
&= \frac{1}{\beta} \log X(x+1,t) - \frac{1}{\beta} \log \frac{X(x+1,t) + X(x-1,t)}{2}\\
&\qquad + \frac{1}{\beta} \log\frac{X(x+1,t) +X(x-1,t) }{2} -  \frac{1}{\beta}\log X(x-1,t) \\
&= \frac{1}{\beta}\log \biggl( 1 + \frac{X(x+1,t) - X(x-1,t)}{X(x+1,t)  +X(x+1,t)} \biggr) - \frac{1}{\beta}\log \biggl(1 - \frac{X(x+1,t) -X(x-1,t)}{X(x+1,t) + X(x-1,t)} \biggr). 
\end{aligned}
\end{equation*}
Using the assumption in equation \eqref{eq:fracdiff}, we can now  expand the logarithm to first order and get~\eqref{eq:controldiff}. Expanding to second order gives \eqref{eq:carefulTaylor}. 
\end{proof}

The purpose of the above lemma is to understand the difference $\fpoly(x+1,t) - \fpoly(x-1,t)$ by understanding the quantities  $X(x+1,t) - X(x-1,t)$ and $X(x+1,t)+X(x-1,t)$. Towards this end, we will prove the following two theorems.

\begin{thm} \label{thm:lowerbnd}
Fix some $\epsilon>0$, $a>0$ and $b>0$. Then  there is an event $\Omega_L$ with $\mathbb{P}(\Omega_L) = 1 - o(1)$ such that on $\Omega_L$, we have the following estimate:
\begin{equation*}
\inf_{|x| \le aN, \, 0 \le t\le bN } X(x,t) \gtrsim N^{-\epsilon}.
\end{equation*}
\end{thm}

\begin{thm} \label{thm:upperbnd}
Fix some $\epsilon >0$, $a>0$ and $b>0$. Then  there is an event $\Omega_U$ with $\mathbb{P}(\Omega_U) = 1- o(1)$, such that on $\Omega_U$, we have the following estimate:
\begin{equation*}
\sup_{|x| \le aN, \, 0 \le t \le bN}|X(x+1,t) -X(x-1,t)| \lesssim N^{-1/4 + \epsilon}. 
\end{equation*} 

\end{thm}

As a corollary of these two theorems and Lemma \ref{lem:fracdiff}, we have the following statement.
\begin{cor} \label{cor:diffcontrol}
Fix some $\epsilon>0$, $a>0$ and $b>0$. Let $\Omega$ be the intersection of the events $\Omega_L$ and $\Omega_U$ from Theorems \ref{thm:lowerbnd} and Theorem \ref{thm:upperbnd}. Then $\mathbb{P}(\Omega) = 1- o(1)$, and on $\Omega$, we have
\begin{equation} \label{eq:gooddiff}
\sup_{|x| \le aN, \, 0 \le t \le bN} |\fpoly(x+1,t) - \fpoly(x-1,t)| \lesssim N^{-1/4 + \epsilon}.
\end{equation}
\end{cor}

\section{Proof of Theorem \ref{thm:lowerbnd}}
It is not hard to see that the function $\fpoly$ is convex in the noise variables $y(x,t)$, which gives us tools for deriving lower tail bounds for $\fpoly(x,t)$. A consequence of this convexity is the following exponential moment estimate.  
\begin{thm} \label{thm:momentbnd}
Fix some $a>0$ and $b>0$. For any $\theta>0$, there exists some constant $C(\theta)$ that does not depend on $N$ (but may depend on $a$ and $b$) such that for all $|x|\le aN$ and $0\le t\le bN$, 
\begin{equation} \label{eq:lowmomentbnd}
\mathbb{E}[\exp( - \theta \fpoly(x,t))] \le C(\theta).
\end{equation}
\end{thm}

Before we prove the theorem, let us give the main corollary of this estimate.

\begin{cor}
For any $(x,t)$ and any $\epsilon>0$, $\mathbb{P}(X(x,t) \le N^{-\epsilon}) = O(N^{-3})$.
\end{cor}
\begin{proof}
Recall that $X(x,t) = \exp[\beta f^{\text{poly}}(x,t)]$. Thus,  \eqref{eq:lowmomentbnd} implies that $\mathbb{E}[ (X(x,t))^{-\theta}]\le C(\theta)$ for arbitrary $\theta>0$. By Markov's inequality, this shows that $\mathbb{P}(X(x,t) \le N^{-\epsilon}) \le C(\theta)/ N^{\theta \epsilon}$. Choosing $\theta = 3/\epsilon$ completes the proof.
\end{proof}

Using the above corollary, and taking union bound over all $x,t$ with $|x| \le aN$ and $0 \le t \le bN$ proves Theorem \ref{thm:lowerbnd}. So it remains to prove Theorem \ref{thm:momentbnd}.

\begin{proof}[Proof of Theorem \ref{thm:momentbnd}]
Observe that we are dealing with the same function as in the proof of \cite[Lemma 6.1]{chatterjee21b1}. Following all the computations in that proof, we see that there are only two new estimates that we need to check in our one dimensional case. Throughout, we take $(x,t)\in [-aN, aN] \times[0,bN]$

Step 1: 
We need to check that $\mathbb{E}[X(x,t)^2]=\mathbb{E}(\mu^{N_t}) \lesssim 1$, where 
\[
\mu := \frac{m(2 \beta N^{-1/4})}{(m(\beta N^{-1/4}))^2},
\]
with $m$ denoting the moment generating function of the noise variables and and $N_t$ is the number of times that two simple symmetric random walks $S_n$ and $S_n'$ started from $0$ intersect before time $t$, excluding time $0$, but including time $t$. With this convention, $N_0=0$. Note that $(S_n - S_n')/2$ has the distribution of a lazy random walk with the probability of not moving equal to $1/2$. Note also that since $m$ is finite in a neighborhood of zero, $\mu$ has the expansion 
 \begin{equation} \label{eq:muexpan}
\mu = 1 + \frac{K_1(\beta)}{N^{1/2}} + \frac{K_2(\beta)}{N^{3/4}} + \frac{K_3(\beta)}{N} +o(N^{-1})
 \end{equation}
 for some constants $K_1(\beta),K_2(\beta),K_3(\beta)$ depending on $ \beta$.
Let $E(z)$ be the generating function for $ \mathbb{E}[\mu^{N_t}]$, defined as 
\begin{equation*}
E(z):= \sum_{s=0}^{\infty} \mathbb{E}[\mu^{N_{s}}] z^{s}.
\end{equation*}
Note that $N_s \le s$, and by the Cauchy--Schwarz inequality, $\mu\ge 1$. Thus, $\mathbb{E}[\mu^{N_s}] \le \mu^s$, and so,  $E(z)$ converges in the open disk of radius $\mu^{-1}$ centered at the origin. 


Let $p(k)$ be the probability that the first return to $0$ after time $0$ of the lazy random walk happens at time $k$ (so that $p(0)=0$).  Define the generating function 
\begin{equation*}
P(z):= \sum_{k=0}^{\infty} p(k) z^k,
\end{equation*}
which converges in the open unit disk. Note that $E$ satisfies the  recursion
\begin{equation*} 
\mathbb{E}[\mu^{N_s}] = \mu \sum_{k=0}^s p(k) \mathbb{E}[\mu^{N_{s-k}}] +  \sum_{k=s+1}^{\infty }p(k), 
\end{equation*}
obtained by by conditioning on the time of first return to $0$. From this, we deduce the following generating function relation:
\begin{equation}\label{eq:ezpz}
E(z) = \mu E(z) P(z) + R(z) +1,
\end{equation}
where $R(z)$ is the remainder generating function 
\begin{equation*}
R(z)= \sum_{s=1}^{\infty} z^s \biggl(\sum_{k=s+1}^{\infty}p(k)\biggr).
\end{equation*}
We now define two new objects. For $n\ge 1$, let $SO_n$ be the probability that the lazy random walk stays strictly above zero from times $1$ to $n-1$, only returning to $0$ at time $n$. For example, $SO_1 = 1/2$ and $SO_2 = 1/16$ (an up step followed by a down step). We adopt the convention that $SO_0 =0$. Let $SO(z)$ be its generating function. Next, let $O_n$ be the probability that the lazy random walk is $0$ at time $n$, and is greater than or equal to $0$ at times $1$ to $n-1$. For example, $O_1 = 1/2$ and $O_2 = (1/2)( 1/2) + (1/4) (1/4) = 5/16$. We adopt the conventions that $O_0 = 1$ and $O_n=0$ for $n<0$.  Let $O(z)$ be its generating function. Note that the generating functions $SO(z)$ and $O(z)$ converge in the open unit disk. 
 
We have the following relations between $SO$ and $O$. For the first relation, observe that if a lazy random walk path stays strictly above $0$ up to time $n-1$ before returning to $0$ at time $n$, then the first step is up and the last step is down. Furthermore, in between the first and the last step, the path stays greater than or equal to $1$. The second relation is obtained by considering the various possible times of the first return to $0$:
\begin{equation*}
\begin{aligned}
& SO_n = \frac{1}{16} O_{n-2} + \frac{1}{2} \delta_{1,n},\\
& O_n = \sum_{i=1}^n SO_{i} O_{n-i} + \delta_{0,n},
\end{aligned}
\end{equation*}  
where $\delta_{x,y}= 1$ if $x=y$ and $0$ otherwise. 
This gives us the equations
\begin{equation}\label{eq:genfuns}
\begin{aligned}
& SO(z) = \frac{1}{16} z^2 O(z) + \frac{z}{2},\\
& O(z) = SO(z) O(z) +1.
\end{aligned}
\end{equation}
By substituting the first equation into the second, we derive the equation,
\begin{equation*}
O(z) = \frac{1}{16}z^2 O(z)^2 + \frac{z}{2} O(z) + 1.
\end{equation*}
We solve this quadratic to get the equation 
\begin{equation*}
O(z) = \frac{1 - \frac{z}{2} - \sqrt{1 - z}}{\frac{z^2}{8}}, 
\end{equation*}
since the above is the only solution of the two solutions to the quadratic that is analytic in a neighborhood of zero and satisfies $O(0) = 1$. Substituting this into the first equation in \eqref{eq:genfuns}, we get
\begin{equation*}
SO(z) = \frac{z}{2}  + \frac{1}{2} \biggl[ 1 - \frac{z}{2} - \sqrt{1-z}\biggr].
\end{equation*}
Now note that 
\[
P(z) =2\biggl[ SO(z) - \frac{z}{2}\biggr] + \frac{z}{2} = 1- \sqrt{1-z},
\]
because we can consider walks that strictly lie above zero and those that strictly lie below zero (aside from the lazy step that stays at $0$). Furthermore, we can express $R(z)$ as
\begin{align*}
R(z) &= \sum_{s=1}^\infty z^s \biggl(1-\sum_{k=0}^s p(k)\biggr)\\
&= \frac{z}{1-z} - \sum_{s=1}^\infty \sum_{k=0}^s p(k) z^k z^{s-k}\\
&= \frac{z}{1-z} - \frac{1}{1-z} P(z) = -1 + \frac{1}{\sqrt{1-z}}.
\end{align*}
Using the above expressions for $P(z)$ and $R(z)$ in equation \eqref{eq:ezpz}, we get
\begin{equation*}
E(z) =  \frac{1+ R(z)}{1- \mu P(z)} = \frac{1}{\sqrt{1-z}}\frac{1}{1 -\mu[1- \sqrt{1-z}]}.
\end{equation*}
By clearing out the square root in the denominator, we can write $E(z)$ as
\begin{equation} \label{eq:Ezexpansion}
\begin{split}
E(z) &=  \frac{1}{\sqrt{1-z}}[(1-\mu) - \mu \sqrt{1-z}] \frac{1}{1- 2\mu + \mu^2 z} \\
&= \frac{ \frac{1-\mu}{\sqrt{1-z}} -\mu}{1- 2\mu} \sum_{k=0}^{\infty} \left( \frac{\mu^2}{1-2 \mu}\right)^k(- z)^k.
\end{split}
\end{equation}
In the last equality, we expanded the denominator as 
\[
\frac{1}{a  +b z} = \frac{1}{a} \sum_{k=0}^{\infty} \left(- \frac{b}{a} z \right)^k.
\]
This power series is convergent in a small radius around $0$, specifically when 
\[
|z| \le \frac{|1 - 2 \mu|}{\mu^2}.
\]
Recall that $1/\sqrt{1-z}$ has the series expansion
\begin{equation*}
 \frac{1}{\sqrt{1-z}} = \sum_{k=0}^\infty {-1/2 \choose k} (-z)^k.
\end{equation*}
Furthermore, we have the estimate  $$ 0 \le {-1/2 \choose k}(-1)^k \le \frac{D}{\sqrt{k}}$$ for some constant $D$. Noting that $$1- \mu = -\frac{K_1(\beta)}{N^{1/2}} + o(N^{-1/2}),$$ $  -1 \ge 1-2\mu \ge -2$, and $\mu^2 \ge 2\mu -1$, and applying the asymptotic expansion of $1/\sqrt{1-z}$ in~\eqref{eq:Ezexpansion}, we see that the coefficient of $z^t$ in the series expansion of $E(z)$, which is just $\E(\mu^{N_t})$, satisfies 
\begin{equation*}
\begin{aligned}
\E(\mu^{N_t}) &\lesssim \frac{1}{N^{1/2}} \sum_{k=1}^t \frac{1}{\sqrt{k}} \biggl( \frac{\mu^2}{2\mu-1}\biggr)^{t-k} + \biggl( \frac{\mu^2}{2\mu -1} \biggr)^t\\& \lesssim \biggl( \frac{\mu^2}{2\mu-1} \biggr)^t \biggl[ 1 + \frac{1}{N^{1/2}} \sum_{k=1}^t \frac{1}{\sqrt{k}} \biggr]\\
& \lesssim \biggl(\frac{\mu^2}{2\mu -1}\biggr)^t \biggl[ 1 + \frac{ t^{1/2}}{N^{1/2}} \biggr]\lesssim \biggl(\frac{\mu^2}{2\mu -1}\biggr)^t,
\end{aligned}
\end{equation*} 
 where in the last line we used the fact that $t = O_n$. 
 Now, by \eqref{eq:muexpan}, we have $$\mu^2=1 + \frac{2 K_1(\beta)}{N^{1/2}} + \frac{2 K_2(\beta)}{N^{3/4}} + \frac{2 K_3(\beta) + K_1^2(\beta)}{N} + o\biggl(\frac{1}{N}\biggr)$$ and $$2\mu-1 = 1 + \frac{2 K_1(\beta)}{N^{1/2}}   + \frac{2K_2(\beta)}{N^{3/4}} + \frac{2 K_4(\beta)}{N} +o\biggl(\frac{1}{N}\biggr).$$
This means that $\mu^2$ and $2\mu-1$ agree with each other to terms up to order larger than $1/N$. Thus, we get that  
\[
\biggl(\frac{\mu^2}{2\mu -1} \biggr)^{t} = \biggl(1 + O\biggl(\frac{1}{N}\biggr)\biggr)^t = O(1),
\]
which completes the proof of the claim that $\E(\mu^{N_t}) \lesssim 1$.

Step 2: We need to show that 
\[
\mathbb{E}|X(x,t) \nabla\fpoly(x,t)|^2 = \frac{\beta^2}{N^{1/2}} \mathbb{E}[N_t \mu^{N_t}]\lesssim 1.
\]
Here, $\nabla \fpoly(x,t)$ is the gradient of $\fpoly$ when considered as a function of the noise variables $\xi$. 
The argument for Step 1 can readily show that 
\[
\ee\biggl[\biggl(1+\frac{1}{\sqrt{N}}\biggr)^{N_t}\biggr]\lesssim 1.
\]
Since  
\[
\biggl(1+\frac{1}{\sqrt{N}}\biggr)^{N_t} = \sum_{k=0}^{N_t}{N_t \choose k} \frac{1}{N^{k/2}},
\]
we get that 
\[
\E(N_t) \lesssim \sqrt{N}, \ \ \ \E\biggl[{N_t \choose 2}\biggr] \lesssim N,
\]
which can be combined to get 
\begin{align}\label{eq:rwintersect}
\E(N_t^2) \lesssim N.
\end{align}
By the Cauchy--Schwarz inequality, this shows that 
\begin{align*}
\E(N_t \mu^{N_t}) &\le \sqrt{\E(N_t^2)\E(\mu^{2N_t})}\lesssim \sqrt{N},
\end{align*}
where the last inequality holds because of the preceding estimate on $\mathbb{E}[N_t^2]$ and the bound $\E(\mu^{2 N_t})\lesssim 1$ using the argument of Step 1, because we can apply Step 1 with $\mu^2 = 1+2K_1(\beta)N^{-1/2} + o(N^{-1/2})$ instead of $\mu$. 

Having proved Steps 1 and 2, the rest of the proof now proceeds exactly as the proof of \cite[Lemma 6.1]{chatterjee21b1}.
\end{proof}

\section{Proof of Theorem \ref{thm:upperbnd}} \label{sec:upperbnd}


The polymer partition $X$ defined in Section \ref{sec:polymer} can be expressed as 
\begin{equation}\label{eq:Xform}
X(x,t):= \frac{1}{2^t} \sum_{S \in RW(x,t)} \prod_{s =1}^t [1 + \xi(S(s),s)]
\end{equation}
where $RW(x,t)$ is the set of all simple symmetric random walk paths on $\Z$ that terminate at $x$ at time $t$. { These random walks do not have a fixed starting point at time $t=0$.} Also, 
\[
\xi(z,s) := \frac{\exp(\beta N^{-1/4} y(z,s))}{m(\beta N^{-1/4})}-1.
\]
Note that \eqref{eq:Xform} is valid for $t=0$ too, if we adopt the usual convention that an empty product equals $1$.

Notice that $\xi(z,s)$ are i.i.d. and have mean $0$.
We have the following result about the growth of the moments of $\xi(z,s)$.
\begin{lmm} \label{lem:momentest}
Define $\xi$ as above. Then for all $N$ large enough (depending only on the $\beta$ and the law of the noise variables) and any $p\ge 1$,
\begin{equation*}
    \mathbb{E}|\xi(z,s)|^{2p} \le \frac{C(p, \beta)}{N^{2p/4}},
\end{equation*}
where $C(p, \beta)$ is a constant that depends on $p$, $\beta$  and the law of the noise variables (and not on~$N$).
\end{lmm}
\begin{proof}
For simplicity, we will write $y$ and $\xi$ instead of $y(z,s)$ and $\xi(z,s)$. Also, throughout, $C$ will denote any constant that depends only on $p$, $\beta$ and the law of the noise variables.

Since $m$ is finite (and hence, continuous) in a neighborhood of zero and $m(0)=1$, it follows that $m(\beta N^{-1/4})$ is uniformly bounded below by a positive constant for all large enough $N$. Thus, if $y'$ denotes an independent copy of $y$, then by a simple application of Jensen's inequality for conditional expectation, we get
\begin{align*}
\E|\xi|^{2p} &\le C \E|\exp(\beta N^{-1/4} y) - m(\beta N^{-1/4})|^{2p}\\
&= C \E|\exp(\beta N^{-1/4} y) - \E(\exp(\beta N^{-1/4}y')|y)|^{2p}\\
&\le C \E|\exp(\beta N^{-1/4} y) - \exp(\beta N^{-1/4} y')|^{2p}.
\end{align*}
Using the inequality 
\[
|e^u - e^v|\le \frac{1}{2}|u-v|(e^u + e^v)
\]
that holds for all $u,v\in \R$, we get from the previous display that
\begin{align*}
\E|\xi|^{2p} &\le C N^{-2p/4}\E(|y-y'|^{2p} (\exp(\beta N^{-1/4}y) + \exp(\beta N^{-1/4}y'))^{2p}).
\end{align*}
It is easy to see that the expectation on the right can be bounded above by a constant that does not depend on $N$.
\end{proof}

For any $x\in \Z$ and $t\in \Z_+\setminus\{0\}$, let
\begin{align*}
\Gamma(x,t) := \frac{1}{2}(X(x+1,t-1) + X(x-1, t-1)).
\end{align*}
Recall that we defined in Section \ref{sec:result} that for each $x\in  \Z$ and $t\in \Z_+$, $p(x,t)$ is the probability that a simple symmetric random walk on $\Z$ started at $0$ at time $0$ ends up at $x$ at time $t$. Recall also that 
\[
\Delta(x,t) = p(x+1,t) - p(x-1,t).
\]
With the above notations, we now describe a `pseudo-chaos expansion' for $X$.  This is similar to the Duhamel formula of \cite[Appendix B]{albertsetal14b}, but we will use it to derive more detailed estimates on differences of $\fpoly$ on immediately adjacent points. 

\begin{lmm} \label{lem:PseudoChaos}
For any $x\in \Z$ and $t\in \Z_+$,
\begin{equation}\label{eq:Xexp}
X(x,t)= 1+ \sum_{z\in \Z} \sum_{s=1}^{t}  p(x-z, t-s) \xi(z,s) \Gamma(z, s).
\end{equation}
\end{lmm}
\begin{proof}
The claim will be proved by induction on $t$. The case $t=0$ is trivial, because $X(x,0)=1$ for all $x$, and the right side in \eqref{eq:Xexp} is $1$ when $t=0$, because the sum on the right is empty. Now fix a time $t>0$ and assume that the formula holds for all times $s < t$. From the expression~\eqref{eq:Xform} for $X(x,t)$, we see that 
\begin{equation}\label{eq:Xstep1}
    X(x,t) = (1+\xi(x,t)) \Gamma(x,t) = \Gamma(x,t) + \sum_{z\in \Z} p(x-z, t-t) \xi(z,t)\Gamma(z,t).  
\end{equation}
If $t=1$, this already proves \eqref{eq:Xexp}, because $\Gamma(x,1)=1$. So, let us assume that $t>1$. 
By the induction hypothesis, we get 
\begin{align*}
\Gamma(x,t) &= \frac{1}{2}(X(x-1,t-1)+X(x+1, t-1))\\
&= 1 + \sum_{z\in \Z}\sum_{s=1}^{t-1}\frac{1}{2}(p(x-1-z, t-1-s)+ p(x+1-z,t-1-s))\xi(z,s) \Gamma(z,s). 
\end{align*}
But note that 
\begin{align*}
\frac{1}{2}(p(x-1-z, t-1-s)+ p(x+1-z,t-1-s)) = p(x-z, t-s). 
\end{align*}
Thus,
\begin{align*}
\Gamma(x,t) &= 1 + \sum_{z\in \Z}\sum_{s=1}^{t-1}p(x-z,t-s) \xi(z,s)\Gamma(z,s).
\end{align*}
Combining this with \eqref{eq:Xstep1} completes the proof. 
\end{proof}
For us, one of the main consequences of Lemma \ref{lem:PseudoChaos} is that $X(x+1,t) - X(x-1,t)$ can be written as,
\begin{equation}\label{eq:xdiff}
X(x+1,t) - X(x-1,t) =  \sum_{z\in \Z}\sum_{s=1}^t \Delta(x-z,t-s) \xi(z,s) \Gamma(z,s).
\end{equation}
Note that this holds also for $t=0$, since then the sum on the right is empty and the left side is zero. 
The following lemma allows us to bound the sizes of the above differences using a martingale approach.
\begin{lmm}\label{lmm:MartBound}
For $0\le s\le t$, let 
\[
M_s := \sum_{r=1}^s \sum_{z\in \Z} \Delta(x-z,t-r) \xi(z,r) \Gamma(z,r),
\]
so that $X(x+1,t)-X(x-1,t) = M_t$. Then we have that for any $p\ge 1$, 
\begin{equation} \label{eq:MartinBound}
    \|M_t\|_{L^p}^2 \le C(p) \sum_{s=1}^t \sum_{z\in \Z} \Delta(x-z,t-s)^2 \|\xi(z,s) \Gamma(z,s)\|^{2}_{L^p},
\end{equation}
where $C(p)$ is a constant that depends only on $p$.
\end{lmm}

\begin{proof}
Let $\mf_s$ be the $\sigma$-algebra generated by all noise variables up to time $s$. Then from the above formula for $M_s$, it is easy to see that
\[
\ee(M_s|\mf_{s-1}) = M_{s-1}.
\]
That is, $\{M_s\}_{0\le s\le t}$ is a mean zero martingale adapted to the filtration $\{\mf_s\}_{0\le s\le t}$, with $M_0\equiv 0$. Thus, by the Burkholder--Davis--Gundy inequality \cite[Theorem 1.1]{burkholderetal72}, for any $p\ge 1$,
\begin{align*}
\ee|M_t|^{p}\le C(p) \ee\biggl[\biggl(\sum_{s=1}^t (M_s-M_{s-1})^2\biggr)^{p/2}\biggr],
\end{align*}
where $C(p)$ denotes a constant that depends only on $p$. The value of $C(p)$ will change from line to line in  the following.

Suppose that $p\ge 2$. Then by the above inequality and Minkowski's inequality,
\begin{align}
\|M_t\|_{L^p}^2 &\le C(p) \biggl\|\sum_{s=1}^t (M_s-M_{s-1})^2\biggr\|_{L^{p/2}}\notag\\
&\le C(p) \sum_{s=1}^t \|(M_s-M_{s-1})^2\|_{L^{p/2}}\notag\\
&=  C(p) \sum_{s=1}^t (\ee|M_s-M_{s-1}|^p)^{2/p}. \label{bdg1}
\end{align}
If $Y_1,\ldots,Y_n$ are independent random variables with mean zero, then again by the Burkholder--Davis--Gundy inequality, we have that for any real numbers $a_1,\ldots, a_n$, and any $p\ge 1$,
\[
\ee\biggl|\sum_{i=1}^n a_i Y_i\biggr|^p \le C(p)\ee\biggl[ \biggl(\sum_{i=1}^n a_i^2 Y_i^2\biggr)^{p/2}\biggr]. 
\]
Conditional on $\mf_{s-1}$, $M_s - M_{s-1}$ is a linear combination of the independent mean zero random variables $(\xi(z,s))_{z\in \zz}$. Thus, the above inequality shows that
\begin{align*}
\ee(|M_s-M_{s-1}|^p |\mf_{s-1}) &\le C(p)\ee\biggl[ \biggl(\sum_y \xi(z,s)^2 \Delta(x-z,t-s)^2 \Gamma(z,s)^2\biggr)^{p/2}\biggl|\mf_{s-1}\biggr]. 
\end{align*}
Taking expected value on both sides and plugging into \eqref{bdg1}, and finally applying Minkowski's inequality, we get
\begin{equation}\label{eq:MartBnd}
\begin{aligned}
\|M_t\|_{L^p}^2 &\le C(p) \sum_{s=1}^t \biggl\|\sum_z \xi(z,s)^2 \Delta(x-z,t-s)^2 \Gamma(z,s)^2\biggr\|_{L^{p/2}}\\
&\le  C(p) \sum_{s=1}^t \sum_z \| \xi(z,s)^2 \Delta(x-z,t-s)^2 \Gamma(z,s)^2\|_{L^{p/2}}\\
&= C(p) \sum_{s=1}^t \sum_z \| \xi(z,s) \Delta(x-z,t-s) \Gamma(z,s)\|_{L^p}^2\\
&= C(p) \sum_{s=1}^t \sum_z \Delta(x-z,t-s)^2 \| \xi(z,s)  \Gamma(z,s)\|_{L^p}^2.
\end{aligned}
\end{equation}
This completes the proof of the lemma.
\end{proof}

The following lemma will be used to control the moments of $\Gamma$ via a control of the moments of $X$.
\begin{lmm} \label{lm:nummoment}
Fix any integer $\theta>0$ and real number $b>0$. There exists some constant $C(\theta,\beta, b)$ depending only on $\theta$,$\beta$, $b$, and the law of the noise variables,  such that for any $x\in \Z$ and $t\le bN$,
\begin{equation*}
\mathbb{E}(X(x,t)^{\theta}) \le C(\theta,\beta,b).
\end{equation*}
\end{lmm}
\begin{proof}
By the formula \eqref{eq:Xform}, we see that 
\begin{equation*}
\ee(X(x,t)^\theta) = \frac{1}{2^{t\theta}}\sum_{S_1,\ldots, S_{\theta}\in RW(x,t)}   \mathbb{E}\biggl(\prod_{j=1}^{\theta} \prod_{s=1}^t \frac{\exp(\beta N^{-1/4} y(S_j(s),s))}{m(\beta N^{-1/4})}\biggr).
\end{equation*}
Fix some $S_1,\ldots, S_\theta$. For each $z\in \Z$ and $s\in \Z_+\setminus\{0\}$, let $n(z,s)$ be the number of $j$ such that $S_j(s) = z$. Then by the independence of the noise variables, we get
\begin{align*}
\mathbb{E}\biggl(\prod_{j=1}^{\theta} \prod_{s=1}^t \frac{\exp(\beta N^{-1/4} y(S_j(s),s))}{m(\beta N^{-1/4})}\biggr) &= \mathbb{E}\biggl[\prod_{z\in \Z} \prod_{s=1}^t \biggl(\frac{\exp(\beta N^{-1/4} y(z,s))}{m(\beta N^{-1/4})}\biggr)^{n(z,s)}\biggr]\\
&= \prod_{z\in \Z} \prod_{s=1}^t \ee\biggl[\biggl(\frac{\exp(\beta N^{-1/4} y(z,s))}{m(\beta N^{-1/4})}\biggr)^{n(z,s)}\biggr].
\end{align*}
If $n(z,s)=0$ or $1$ for some $(z,s)$, then the expectation on the right is $1$. Otherwise,  we can write the expectation as
\begin{align*}
    \mathbb{E}[(\xi(z,s) +1)^{n(z,s)}] &\le [\mathbb{E}((\xi(z,s) +1)^{2n(z,s)})]^{1/2} \\
    &= \biggl( 1+ \sum_{k=1}^{2n(z,s)}{2n(z,s) \choose k} \mathbb{E}(\xi(z,s)^{k}) \biggr)^{1/2} = 1 + O(N^{-1/2}), 
\end{align*}
where we used our estimates for the moments of $\xi(z,s)$ from Lemma \ref{lem:momentest}, as well as the fact that $\mathbb{E}(\xi(z,s))=0$. Noting that $n(z,s) \le \theta$, the implicit constant in the $O(N^{-1/2})$ term can be bounded by $C(\theta,\beta)$; this is a constant that only depends on $\theta$, $\beta$, and the law of the noise variables. 

Also, the number of $(z,s)$ such that $n(z,s)>1$ is at most $\sum_{i< j} |S_i\cap S_j|$, where $S_i\cap S_j$ denotes the set $\{1\le s\le t: S_i(s)=S_j(s)\}$. Combining all of these observations, we get
\begin{align*}
\ee(X(x,t)^\theta) &\le \frac{1}{2^{t\theta}}\sum_{S_1,\ldots, S_{\theta}\in RW(x,t)}\biggl(1+\frac{C(\theta, \beta)}{\sqrt{N}}\biggr)^{\sum_{i<j}|S_i\cap S_j|}.
\end{align*}
An application of H\"older's inequality shows that the right side is bounded above by
\begin{align*}
\frac{1}{4^t}\sum_{S_1,S_2\in RW(x,t)}\biggl(1+\frac{C(\theta, \beta)}{\sqrt{N}}\biggr)^{{\theta \choose 2}|S_1\cap S_2|}.
\end{align*}
The analysis in the proof of Step 1 of Theorem \ref{thm:momentbnd} shows that the above quantity is bounded above by a constant that depends only on $\theta$, $\beta$, $b$, and the law of the noise variables.
\end{proof}

The following Lemma gives us control on the $L^2$ sum of $\Delta$.
\begin{lmm} \label{lmm:DeltSum}
There is a constant $C$ such that for any $t\ge 1$, 
\begin{equation*}
    \sum_{z\in \Z}  \Delta(z,t)^2 \le Ct^{-3/2}.
\end{equation*}
\end{lmm}
\begin{proof}
Throughout this proof, $C, C_1,C_2$ will denote universal constants whose values may change from line to line. The value of $C$ may change from line to line. Note that $\Delta(z,t) = 0$ if $|z|>t+1$ or $z$ and $t$ have the same parity. If $z=t+1$ or $z = -t-1$, then $|\Delta(z,t)| = 2^{-t}$, which proves the claim. So, let us henceforth assume that $|z|\le t-1$, and that $z$ and $t$ do not have the same parity. Then  
\begin{align*}
\Delta(z, t) &= \frac{1}{2^t} {t \choose (t+z+1)/2} - \frac{1}{2^t} {t \choose (t+z-1)/2} \\
&= \frac{1}{2^t} {t \choose (t+z-1)/2} \biggl(\frac{(t-z -1)/2 }{(t+z+1)/2 } - 1 \biggr)\\
&= -\frac{1}{2^t} {t \choose (t-z+1)/2} \frac{z}{(t+z+1)/2}. 
\end{align*}
Now, by standard facts about binomial coefficients, 
\begin{align*}
\frac{1}{2^t} {t \choose t/2+z} &\le \frac{1}{2^t} {t \choose \lfloor t/2\rfloor}\le C t^{-1/2}.
\end{align*}
By the previous two displays, we see that if $|z|\le t/2$, then
\[
\Delta(z,t)^2\le \frac{1}{2^t} {t \choose (t+z-1)/2} Ct^{-1/2}\frac{z^2}{t^2},
\]
which implies that
\begin{align*}
\sum_{|z|\le t/2}\Delta(z,t)^2 &\le C t^{-5/2}\sum_{-t+1\le z\le t+1} \frac{1}{2^t} {t \choose (t+z-1)/2} z^2\le Ct^{-3/2},
\end{align*}
where the last inequality holds because the sum on the right is the expected value of the square of a sum of $t$ i.i.d.~Rademacher random variables, which is of order $t$. Next, if $|z|> t/2$, then a simple application of Stirling's formula shows that $p(z\pm 1,t)$ are exponentially small in $t$, and hence
\[
\sum_{|z|> t/2}\Delta(z,t)^2 \le C_1e^{-C_2t}.
\]
Combining the two estimates, we get the desired bound.
\end{proof}


We now have all the tools to complete the proof of Theorem \ref{thm:upperbnd}. Take any $p\ge 1$ and $(x,t)\in [-aN,aN]\times[0,bN]$. By the Cauchy--Schwarz inequality,
\begin{equation} \label{eq:CauMart}
    \mathbb{E}|\xi(x,t) \Gamma(x,t)|^{p} \le (\mathbb{E}|\xi(x,t)|^{2p})^{1/2} (\mathbb{E}(\Gamma(x,t)^{2p}))^{1/2}.
\end{equation}
By Lemma \ref{lem:momentest} amd Lemma \ref{lm:nummoment}, the right side is bounded by $N^{-2p/4}$ times a constant that has no dependence on $N$, $x$ and $t$. Therefore, by Lemma \ref{lmm:MartBound},
\begin{align*}
\|X(x+1, t-1)-X(x-1,t-1)\|_{L^p}^2 &\lesssim N^{-1/2}\sum_{z\in \Z}\sum_{s=1}^t \Delta(x-z, t-s)^2.
\end{align*}
By Lemma \ref{lmm:DeltSum}, the sum on the right is bounded above by a constant. Thus, for any $p\ge 1$, 
\[
\|X(x+1, t-1)-X(x-1,t-1)\|_{L^p} \lesssim O( N^{-1/4}).
\]
Using Markov's inequality with sufficiently large $p$, and taking a union bound over $(x,t)\in [-aN, aN]\times [0,bN]$ now completes the proof of Theorem \ref{thm:upperbnd}.

\section{Further estimates for polymer growth} \label{sec:furtherestim}



In this section we will obtain some further technical estimates for the polymer surface. Recall the event $\Omega$ from Corollary \ref{cor:diffcontrol}. By equation \eqref{eq:carefulTaylor} in Lemma \ref{lem:fracdiff}, and the conclusions of Theorems \ref{thm:lowerbnd} and \ref{thm:upperbnd}, we have that on the event $\Omega$,
\begin{equation}\label{eq:fourthpower}
    ( \fpoly(x+1,t) -  \fpoly(x-1,t))^4 = \frac{16}{\beta^4}\left(\frac{X(x+1,t) - X(x-1,t)}{X(x+1,t) +X(x-1,t)} \right)^4 + O(N^{-1-\epsilon})
\end{equation}
uniformly on a rectangle $[-aN,aN]\times[0,bN]$. 

We will now try understand the ratio on the right side above. Throughout, we will freely use the notations introduced earlier. We start with the following Lemma.
\begin{lmm} \label{lem:timediff}
Fix some $a>0$, $b>0$ and $\delta> 0$. Then there is some event $\Omega_{\Gamma}$ with $\mathbb{P}(\Omega_{\Gamma}) = 1- o(1)$, on which we have that for all $(x,t) \in [-aN, aN] \times [0,bN]$, 
\begin{equation*}
    |\Gamma(x,t) - \Gamma(x\pm 1, t-1)| \lesssim N^{-1/4 + \delta},
\end{equation*}
\end{lmm}
\begin{proof}
Due to the similarity of the proofs, let us consider only the difference $\Gamma(x,t) - \Gamma(x+1,t-1)$. 
Assume that we are in the event $\Omega_U$ from Theorem \ref{thm:upperbnd}. From \eqref{eq:Xform} and the definition of $\Gamma$, we have 
\begin{equation*}
    X(x,t) - \Gamma(x,t) = \xi(x,t) \Gamma(x,t).
\end{equation*}
Using equation \eqref{eq:CauMart}, Lemma \ref{lem:momentest} and Lemma \ref{lm:nummoment} with a sufficiently large $p$, and applying Markov's inequality and a union bound, we can assert that $|\xi(x,t) \Gamma(x,t)| \lesssim N^{-1/4 + \delta}$ on some set $\tilde{\Omega}_1$ with $\mathbb{P}(\tilde{\Omega}_1) = 1 - o(1)$. On $\Omega_U$, we also know that $|X(x,t) - X(x-2,t)| \le N^{-1/4 + \delta}$. Thus, on $\Omega_U \cap \tilde{\Omega}_1$, we have \begin{equation*}
    |\Gamma(x-1,t+1) - \Gamma(x,t)| \le |X(x,t) - \Gamma(x,t)| + \frac{1}{2}|X(x-2,t) - X(x,t)| \lesssim N^{-1/4 + \delta}.
\end{equation*}
We let $\Omega_{\Gamma}$ be the intersection $\Omega_U \cap \tilde{\Omega}_1$.
\end{proof}

\begin{lmm} \label{lem:diffexpand}
Fix some $a>0$, $b>0$ and $\epsilon \in (0,1/100)$.  Then there is some event $\Omega_{X}$ such that $\mathbb{P}(\Omega_{X}) = 1 -o(1)$ and on $\Omega_X$, we have that for all $(x,t)\in [-aN, aN]\times[0,bN]$, 
\begin{equation*}
    \frac{X(x+1,t) - X(x-1,t)}{X(x+1,t) + X(x-1,t)} = \frac{1}{2}\sum_{z\in \Z} \sum_{t- N^{\epsilon} \le s \le t}  \xi(z,s) \Delta(x-z,t-s) + O(N^{-1/4-\epsilon/16}),
\end{equation*}
where the $O$ term is uniform over $(x,t)$ in the above region. 
\end{lmm}

\begin{proof}
Recall from equation \eqref{eq:xdiff} of Section \ref{sec:upperbnd} that
\begin{equation*}
    X(x+1,t) - X(x-1,t) = \sum_{z\in \Z}\sum_{s=1}^t \Delta(x-z,t-s) \xi(z,s) \Gamma(z,s).
\end{equation*}
The proof technique of Lemma \ref{lmm:MartBound} allows us to bound, for any $p\ge 1$, 
\begin{align*}
    &\biggl\|\sum_{z\in \Z} \sum_{1\le s < t - N^{\epsilon}} \Delta(x-z,t-s) \xi(z,s) \Gamma(z,s)\biggr\|_{L^p}^2\\
    &\le C(p) \sum_{1\le s< t- N^{\epsilon}} \sum_z \Delta(x-z,t-s)^2 \|\xi(z,s) \Gamma(z,s)\|_{L^p}^2.
\end{align*}
From the proof of Theorem \ref{thm:upperbnd}, recall that $\|\xi(z,s) \Gamma(z,s)\|_{L^p}^2 \lesssim N^{-1/2}$. Lastly, by  Lemma \ref{lmm:DeltSum},
\begin{equation*}
\begin{aligned}
    \sum_{1\le s < t- N^{\epsilon}} \sum_z \Delta(x-z,t-s)^2  \lesssim  N^{-\epsilon/2}.
\end{aligned}
\end{equation*}
Combining all of the above, using Markov's inequality with large enough $p$, and applying a union bound, we see that there is an event $\tilde{\Omega}_B$ with $\mathbb{P}(\tilde{\Omega}_B) = 1 - o(1)$ on which 
\begin{equation} \label{eq:beforetildeOm}
    X(x+1,t) -X(x-1,t) = \sum_{z\in \Z} \sum_{t-N^{\epsilon}\le s \le t} \Delta(x-z,t-s) \xi(z,s) \Gamma(z,s) + O(N^{-1/4 - \epsilon/8})
\end{equation}
uniformly over all $(x,t)$ in our region. On the set $\Omega_L$ from Theorem \ref{thm:lowerbnd}, we have that $X(x+1,t) + X(x-1,t) \gtrsim N^{-\epsilon/16}$ uniformly of $(x,t)$. Combining this with the above, we have that on the event $\tilde{\Omega}_B \cap \Omega_L$, 
\begin{align} 
    &\frac{X(x+1,t) - X(x-1,t)}{X(x+1,t) + X(x-1,t)} \notag \\ 
    &= \sum_{z\in \Z} \sum_{t - N^{\epsilon} \le s \le t} \Delta(x-z,t-s) \xi(z,s) \frac{\Gamma(z,s)}{2\Gamma(x,t+1)} + O(N^{-1/4-\epsilon/16}). \label{eq:beforeOmegaB}
\end{align}
Now recall that  $|\Gamma(z,s) - \Gamma(z\pm 1, s-1)| \lesssim N^{-1/4 + \epsilon}$
for all $(z,s) \in[-aN,aN] \times[0,bN] $ on the event $\Omega_{\Gamma}$ from Lemma \ref{lem:timediff}. This implies  that for any $(z,s)$ such that $\Delta(x-z,t-s) \ne 0$ and $t - N^{\epsilon}  \le s \le t$, we have 
\[
|\Gamma(z,s) - \Gamma(x,t+1)| \lesssim N^{-1/4 + 2 \epsilon}.
\]
Also recall that on the event $\Omega_L$, we have $\Gamma(x,t) \gtrsim N^{-\epsilon/16}$ uniformly over $(x,t)$. Thus, on $\Omega_L\cap \Omega_\Gamma$, we have 
\[
\frac{\Gamma(z,s)}{\Gamma(x,t+1)} = 1 + O(N^{-1/4 + 3 \epsilon}).
\]
uniformly over all $x$, $z$, $s$ and $t$ such that $(x,t)\in [-aN, aN]\times[0,bN]$,  $\Delta(x-z,t-s) \ne 0$ and $t - N^{\epsilon}  \le s \le t$.

Lemma \ref{lem:momentest} shows that there is an event $\Omega_B$ with $\P(\Omega_B) = 1-o(1)$ on which every $|\xi(z,s)|$, for $(z,s) \in [-aN,aN] \times[0,bN]$, is  $\lesssim N^{-1/4 + \epsilon}$. 
Furthermore, note that for any $s$, 
$$
    \sum_{z} |\Delta(x-z,t-s)| \le \sum_{z} (p(x-z+1,t-s) + p(x-z-1,t-s)) \le 2,
$$
and therefore,
$$
\sum_{z\in \Z} \sum_{t- N^{\epsilon} \le s \le t} |\Delta(x-z,t-s)| \le 2 N^{\epsilon}.
$$
Thus, on $\Omega_L\cap \Omega_\Gamma\cap \Omega_B$, 
\begin{align*}
    &\sum_{z\in \Z} \sum_{ t- N^{\epsilon} \le s \le t} \Delta(x-z,t-s) \xi(z,s) \frac{\Gamma(z,s)}{\Gamma(x,t+1)} \\
    &= \sum_{z\in \Z} \sum_{t- N^{\epsilon} \le s \le t} \Delta(x-z,t-s) \xi(z,s) + O(N^{-1/2 + 5\epsilon}).
\end{align*}
Thus, we can finally define $\Omega_{X}= \tilde{\Omega}_B \cap \Omega_L \cap \Omega_{\Gamma} \cap \Omega_B$ to finish the proof. 
\end{proof}

Our next lemma gives an upper bound on the size of the right side in Lemma \ref{lem:diffexpand}.
\begin{lmm} \label{lem:Kbnd}
Fixing some $\epsilon >0$, define 
$$K(x,t):= \frac{1}{2}\sum_{z\in \Z} \sum_{ t- N^{\epsilon} \le s \le t} \Delta(x-z,t-s) \xi(z,s).$$ Fix some $a>0$, $b>0$, and  $\delta>0$. Then there is an event $\Omega_K$ with  $\mathbb{P}(\Omega_K) = 1- o(1)$ on which we have for all $(x,t)\in [-aN, aN]\times[0,bN]$ that
\begin{equation*}
    K(x,t) \lesssim N^{-1/4 + \delta}.
\end{equation*}
Furthermore, we also have that
\begin{equation}\label{eq:Kmtbnd}
     \|K(x,t)\|_{L^p}^2 \lesssim N^{-1/2}.
\end{equation}
\end{lmm}
(Note that $K$ depends on $\epsilon$, but we prefer to write $K$ instead of $K_\epsilon$ to lighten notation.)
\begin{proof}
Since $K(x,t)$ is a sum of independent terms, we can apply the Burkholder--Davis--Gundy inequality (as in the proof of Lemma \ref{lmm:MartBound}) to get
\begin{equation*}
    \|K(x,t)\|_{L^p}^2 \lesssim \sum_{z\in \Z} \sum_{ t- N^{\epsilon} \le s \le t} \Delta(x-z,t-s)^2 \|\xi(z,s)\|_{L^p}^2.
\end{equation*}
From Lemma \ref{lem:momentest}, we know that $\|\xi(z,s)\|_{L^p}^2 \lesssim N^{-1/2}$, and by Lemma \ref{lmm:DeltSum},  
\[
\sum_{z\in \Z} \sum_{ t- N^{\epsilon} \le s \le t}\Delta(x-z,t-s)^2 \lesssim 1.
\]
We can now Markov's inequality with sufficiently large $p$ and take a union bound to complete the proof. 
\end{proof}

The following corollary is the main result of this section.
\begin{cor} \label{cor:Xd4}
Fix some $a>0$, $b>0$, and $\epsilon \in (0,1/100)$. Then there is some event $\Omega_{4}$ with $\P(\Omega_4)=1-o(1)$ such that on $\Omega_{4}$, we have that for all $(x,t)\in [-aN, aN] \times[0,bN]$, 
\begin{equation*}
\biggl(\frac{X(x+1,t) - X(x-1,t)}{X(x+1,t) + X(x-1,t)}\biggr)^4 = K(x,t)^4 + O(N^{-1-\epsilon/32}).
\end{equation*}
\end{cor}
\begin{proof}
We can bound $|K(x,t)| \lesssim N^{-1/4 + \epsilon/100}$ for all $(x,t)\in [-aN, aN] \times[0,bN]$ with high probability on $\Omega_K$, as in Lemma \ref{lem:Kbnd}.  Then applying Lemma \ref{lem:diffexpand} and taking the fourth power completes the proof.
\end{proof}


\section{The main argument} \label{sec:mainarg}
In this section, we will carry out the most important step in our proof of Theorem \ref{thm:mainresult}, which is to relate the polymer surface $\fpoly$ with the function $f$ defined in equation \eqref{eq:fxtdef} from Section~\ref{sec:prelim} (after fixing $N$). 
First, recall the constant $c$ defined in equation \eqref{eq:cdef}, and after fixing some $\epsilon \in (0,1/100)$, define the `renormalization term' 
\begin{equation}\label{eq:ydef}
    Y(x,t) := \frac{16c}{\beta^4} \sum_{z\in \Z}\sum_{s=1}^t p(x-z,t-s) K(z,s)^4.
\end{equation}
The main result of this section (Theorem \ref{thm:errorprop2}) is that with high probability, 
\[
f(x,t) = \fpoly(x,t) + Y(x,t) + o(1)
\]
for all $(x,t)$ in a given region of the form $[-aN, aN] \times[0,bN]$. To prove this result, we need the following crucial lemma about $Y$.
\begin{lmm} \label{lem:Ydiff}
Fix some $a>0$, $b>0$, and $\epsilon \in (0,1/100)$. Then there is some event $\Omega_{Y}$ with $\P(\Omega_Y)=1-o(1)$ such that on $\Omega_{Y}$, we have that for all $(x,t)\in [-aN, aN]\times[0,bN]$, 
\begin{equation*}
    |Y(x+1,t) - Y(x-1,t)| \lesssim N^{-1+7\epsilon}\le N^{-3/4 - 2\epsilon}.
\end{equation*}
\end{lmm}
\begin{proof}
First note that 
\begin{equation*}
Y(x+1,t) - Y(x-1,t) = \frac{16c}{\beta^4} \sum_{z\in \Z}\sum_{s=1}^t \Delta(x-z,t-s) K(z,s)^4
\end{equation*}
Given $s$, the mean $\mathbb{E}[K(z,s)^4]= \overline{m}(s)$ has no dependence on $z$. Since 
\[
\sum_{z\in \Z} \Delta(x-z,t-s) =0,
\]
the above expression can be written as
\begin{equation*}
Y(x+1,t) - Y(x-1,t) = \frac{16 c}{\beta^4} \sum_{z\in \Z}\sum_{s=1}^t \Delta(x-z,t-s) (K(z,s)^4 - \overline{m}(s)).
\end{equation*}
Now, define $Y^{(l_1,l_2)}$ as
\begin{equation*}
Y^{(l_1,l_2)} := \frac{16 c}{\beta^4} \sum_{\substack{z\ \equiv\  l_1 \ (\mathrm{mod} \ 2 \lceil N^{\epsilon} \rceil)\\ s\ \equiv \ l_2 \ (\mathrm{mod} \ \lceil N^{\epsilon} \rceil), \, s\le t}} \Delta(x-z,t-s) (K(z,s)^4 - \overline{m}(s)).
\end{equation*} 
The point of the introduction of $Y^{(l_1,l_2)}$ is that it is  a sum of independent terms. It is clear that
\begin{equation*}
Y(x+1,t) - Y(x-1,t) = \sum_{(l_1,l_2)} Y^{(l_1,l_2)}.
\end{equation*}
Thus, for any $p\ge 1$,
\begin{align*}
|Y(x+1,t) - Y(x-1,t)|^{2p} &\le (N^{2\epsilon} \max_{(l_1,l_2)}|Y^{(l_1,l_2)}|)^{2p}\\
&\le N^{4\epsilon p}\sum_{(l_1,l_2)}|Y^{(l_1,l_2)}|^{2p}.
\end{align*}
Thus, we get
\begin{equation*}
\mathbb{E}|Y(x+1,t) - Y(x-1,t)|^{2p} \le N^{6\epsilon p} \max_{(l_1,l_2)} \mathbb{E}|Y^{(l_1,l_2)}|^{2p}.
\end{equation*}
Applying the Burkholder--Davis--Gundy inequality and Minkowski's inequality as in the proof of Lemma \ref{lmm:MartBound}, we have
\begin{equation*}
(\mathbb{E}|Y^{(l_1,l_2)}|^{2p})^{1/p} \lesssim  \sum_{\substack{z\ \equiv\  l_1 \ (\mathrm{mod} \ 2 \lceil N^{\epsilon} \rceil)\\ s\ \equiv \ l_2 \ (\mathrm{mod} \ \lceil N^{\epsilon} \rceil), \, s\le t}} \Delta(x-z,t-s)^2(\mathbb{E}|K(z,s)^4 - \overline{m}(s)|^{2p})^{1/p}.
\end{equation*}
Equation \eqref{eq:Kmtbnd} in Lemma \ref{lem:Kbnd} shows that $\mathbb{E}|K(y,s)|^{8p} \lesssim N^{-2p}$. Similarly, $\overline{m}(s) = O(N^{-1})$. Thus, $\mathbb{E}|K(y,s) - \overline{m}(s)|^{2p}\lesssim N^{-2p}$.  Furthermore, by Lemma \ref{lmm:DeltSum}, 
\[
\sum_{z\in \Z}\sum_{s=1}^t|\Delta(x-z,t-s)|^2 \lesssim 1.
\]
Combining these facts shows that
$$
\mathbb{E}|Y(x+1,t) - Y(x-1,t)|^{2p}\lesssim N^{6\epsilon p} N^{-2p},
$$
and we can get our desired event $\Omega_{Y}$ by taking high enough moments and applying Markov's inequality and a union bound.
\end{proof}

We now arrive at the main result of this section. 
\begin{thm} \label{thm:errorprop2}
Fix some $a>0$, $b>0$ and $\epsilon \in (0,1/100)$. Let $Y(x,t)$ be defined as above. Then there is an event $\Omega_2$ with $\P(\Omega_2)=1-o(1)$, such that on $\Omega_2$, we have that for all $(x,t)\in [-aN, aN]\times[0,bN]$, 
\[
f(x,t) = \fpoly(x,t) + Y(x,t) + o(1),
\]
where the $o(1)$ term is uniform in $(x,t)$.
\end{thm}
\begin{proof}
Let $\Omega_2:= \Omega \cap \Omega_{Y} \cap \Omega_{4}$, where $\Omega$ is from Corollary \ref{cor:diffcontrol}, $\Omega_Y$ is from Lemma \ref{lem:Ydiff}, and $\Omega_4$ is from Corollary \ref{cor:Xd4}. We will prove the claim for all $(x,t)$ satisfying the constraint 
\begin{align}\label{eq:xtrestriction}
|x| + |t|\le \min\{aN,bN\}.
\end{align}
Note that any such $(x,t)$ is automatically in $[-aN, aN]\times[0,bN]$, but the converse is not true. The remaining points in this rectangle can be handled simply by repeating the whole argument with $a$ and $b$ both replaced by $a+b$ (and replacing the event $\Omega_2$ by the corresponding event for the rectangle $[-(a+b)N, (a+b)N]\times[-(a+b)N, (a+b)N]$). 

Define
\[
\delta(x,t) := f(x,t)-\fpoly(x,t) - Y(x,t).
\]
We will prove by induction on $t$ that on $\Omega_2$, for all $(x,t)$ satisfying \eqref{eq:xtrestriction}, we have 
\begin{align}\label{eq:induc}
|\delta(x,t)| &\le N^{-1-\epsilon/2} t,     
\end{align}
provided that  $N\ge N_0$, where $N_0$ is a deterministic threshold depending only on $a$, $b$, $\epsilon$, $\psi$, and the law of the noise variables. We will choose $N_0$ later.

Throughout, we will work under the assumption that $\Omega_2$ holds. Fix some $(x,t)$ satisfying \eqref{eq:xtrestriction}, and assume that \eqref{eq:induc} has been proved for up to time $t-1$. (Note that~\eqref{eq:induc} holds trivially when $t=0$, since $\delta(x,0)=0$ for all $x$.) Define
\begin{align*}
&f_1 := f(x-1,t-1), \ \ f_2 := f(x+1, t-1),\\
&\fpoly_1 := \fpoly(x-1,t-1), \ \ \fpoly_2 := \fpoly(x+1, t-1),\\
&Y_1 := Y(x-1,t-1), \ \ Y_2 := Y(x+1,t-1),\\
&\delta_1 := \delta(x-1,t-1), \ \ \delta_2 := \delta(x+1,t-1).
\end{align*}
Now, if $(x,t)$ satisfies \eqref{eq:xtrestriction}, so does $(x-1,t-1)$ and $(x+1,t-1)$. Thus, by the induction hypothesis, $|\delta_1|$ and $|\delta_2|$ are bounded above by $N^{-1-\epsilon/2} (t-1)$, and so,  
\begin{align}\label{eq:delta12}
\max\{|\delta_1|, |\delta_2|\}  \le bN^{-\epsilon/2}.
\end{align}
Next, note that by Corollary \ref{cor:diffcontrol}, Lemma \ref{lem:Ydiff} and the above display, we have 
\begin{align*}
|f_1-f_2| &= |\fpoly_1 + Y_1 +\delta_1 - (\fpoly_2 + Y_2 + \delta_2)|\\
&\le |\fpoly_1 - \fpoly_2| + |Y_1-Y_2| + |\delta_1| + |\delta_2| \lesssim N^{-\epsilon/2}.
\end{align*}
By \eqref{eq:frec2} and \eqref{eq:phiderivs}, and the assumption that $\psi$ is $C^6$ in a neighborhood of the origin, this allows us to apply Taylor expansion to deduce that
\begin{align}
f(x,t) &= \frac{f_1+f_2}{2} + \frac{\beta}{8}(f_1-f_2)^2 + \biggl(-\frac{\beta^3}{192}+c\biggr)(f_1-f_2)^4 \notag \\
& \qquad \qquad + C(x,t)(f_1-f_2)^6 + \frac{1}{N^{1/4}} y(x,t) - \frac{\log m(\beta N^{-1/4})}{\beta}, \label{eq:fbound}
\end{align}
where $C(x,t)$ is a number that depends on $x$, $t$ and the particular realization of the noise variables, satisfying 
\begin{align}\label{eq:cxt}
|C(x,t)|\lesssim 1,
\end{align}
provided that $N\ge N_1$, where $N_1$ is a deterministic threshold depending only on $a$, $b$, $\epsilon$, $\psi$, and the law of the noise variables. The fact that $|C(x,t)| \lesssim 1$ comes from the assumption that $\phi$ is in $C^6$ so $\phi^{(6)}$ will be finite in a compact interval around $0$. We are computing $\phi(f_1 - f_2)$ where $f_1 - f_2 \lesssim N^{-\epsilon/2}$ by our inductive hypothesis.Thus, we can safely apply the Taylor expansion and use the fact that $\phi^{(6)}$ is bounded in a compact neighborhood of $0$.

Similarly, by Corollary \ref{cor:diffcontrol}, Lemma \ref{lem:Ydiff}, and Taylor expansion using \eqref{eq:polyrec} and \eqref{eq:polyderivs}, we have  
\begin{equation} 
\begin{aligned} 
&\biggl|\fpoly(x,t) - \biggl(\frac{\fpoly_1 + \fpoly_2}{2} + \frac{\beta}{8} (\fpoly_1-\fpoly_2)^2  -\frac{\beta^3}{192}(\fpoly_1-\fpoly_2)^4 \notag \\
& \qquad \qquad + \frac{1}{N^{1/4}} y(x,t) - \frac{\log m(\beta N^{-1/4})}{\beta}\biggr)\biggr|\notag \\
&\lesssim (\fpoly_1-\fpoly_2)^6 \lesssim N^{-3/2 + 6\epsilon}\lesssim N^{-1-\epsilon}. \label{eq:polybound}
\end{aligned}
\end{equation}
But by Corollary \ref{cor:diffcontrol} and Lemma \ref{lem:Ydiff},
\begin{align}
&|(\fpoly_1-\fpoly_2)^2  - (\fpoly_1+Y_1-\fpoly_2-Y_2)^2| \notag \\
&\le |Y_1-Y_2| |2(\fpoly_1-\fpoly_2) + Y_1-Y_2|\lesssim N^{-1-\epsilon}, \label{eq:poly1}
\end{align}
and similarly,
\begin{align}\label{eq:poly2}
&|(\fpoly_1-\fpoly_2)^4  - (\fpoly_1+Y_1-\fpoly_2-Y_2)^4| \lesssim N^{-1-\epsilon}, 
\end{align}
and furthermore,
\begin{align}\label{eq:poly21}
&(\fpoly_1 + Y_1 - \fpoly_2 - Y_2)^6 \lesssim N^{-1-\epsilon}.
\end{align}
Finally, note that by Corollary \ref{cor:Xd4} and equations~\eqref{eq:fourthpower} and \eqref{eq:poly2},
\begin{equation}
\begin{aligned} 
Y(x,t) - \frac{Y_1+Y_2}{2} &= \frac{16}{\beta^4}cK(x,t)^4\notag \\
&= \frac{16}{\beta^4} c \biggl(\frac{X_1-X_2}{X_1+X_2}\biggr)^4 + O(N^{-1-\epsilon})\notag \\
&= c(\fpoly_1 - \fpoly_2)^4 + O(N^{-1-\epsilon})\notag \\
&= c(\fpoly_1+Y_1-\fpoly_2-Y_2)^4 + O(N^{-1-\epsilon}).\label{eq:poly3}
\end{aligned}
\end{equation}
Using \eqref{eq:cxt}, \eqref{eq:poly1}, \eqref{eq:poly2},  and \eqref{eq:poly21}  in \eqref{eq:polybound}, we get
\begin{align*}
&\biggl|\fpoly(x,t) + Y(x,t) - \biggl(\frac{\fpoly_1+ Y_1 + \fpoly_2+Y_2}{2} \\
&\qquad + \frac{\beta}{8} (\fpoly_1+Y_1-\fpoly_2-Y_2)^2  + \biggl(-\frac{\beta^3}{192} + c\biggr)(\fpoly_1+Y_1-\fpoly_2-Y_2)^4 \notag \\
& \qquad + C(x,t)(\fpoly_1 + Y_1 - \fpoly_2 - Y_2)^6 + \frac{1}{N^{1/4}} y(x,t) - \frac{\log m(\beta N^{-1/4})}{\beta}\biggr)\biggr| \lesssim N^{-1 - \epsilon}. 
\end{align*}
Combining this with \eqref{eq:fbound} (and recalling that $f_i=\fpoly_i + Y_i + \delta_i$ for $i=1,2$), we get
\begin{align*}
|\delta(x,t)| &\le \biggl|\frac{\delta_1+\delta_2}{2} \\
&\qquad + \frac{\beta}{8} ((\fpoly_1 + Y_1 + \delta_1 - \fpoly_2 - Y_2 - \delta_2)^2 - (\fpoly_1 + Y_1 - \fpoly_2-Y_2)^2)\\
&\qquad + \biggl(-\frac{\beta^3}{192} + c\biggr)((\fpoly_1+Y_1+\delta_1-\fpoly_2-Y_2-\delta_2)^4 \\
&\qquad \qquad \qquad - (\fpoly_1+Y_1-\fpoly_2-Y_2)^4)\\
&\qquad + C(x,t)((\fpoly_1 + Y_1 + \delta_1 - \fpoly_2 - Y_2 - \delta_2)^6\\
&\qquad \qquad \qquad - (\fpoly_1 + Y_1 - \fpoly_1 - Y_2)^6)\biggr| + O(N^{-1-\epsilon}).
\end{align*}
Notice that the term inside the absolute values on the right can be written as
\begin{align*}
\frac{\delta_1+\delta_2}{2} + (\delta_1-\delta_2) B = \delta_1\biggl(\frac{1}{2} + B\biggr) + \delta_2\biggl(\frac{1}{2}- B\biggr),
\end{align*}
where $B$ is a quantity which, by Corollary \ref{cor:diffcontrol}, Lemma \ref{lem:Ydiff}, inequality~\eqref{eq:cxt}, and inequality~\eqref{eq:delta12}, is less than $1/4$, provided that $N\ge N_2$, where $N_2$ is a deterministic threshold depending only on $a$, $b$, $\epsilon$, $\psi$, and the law of the noise variables. So, if $N\ge \max\{N_1, N_2\}$, we have that
\[
|\delta(x,t)|\le \max\{|\delta_1|, |\delta_2|\} + O(N^{-1-\epsilon}).
\]
Finally, note that the $O(N^{-1-\epsilon})$ error term is bounded above by $N^{-1-\epsilon/2}$ if $N\ge N_3$ where $N_3$ is a deterministic threshold depending only on $a$, $b$, $\epsilon$, $\psi$, and the law of the noise variables. Thus, if our original choice of $N_0$ is $\max\{N_1, N_2, N_3\}$, the induction step goes through, completing the proof. 
\end{proof}


\section{Concentration of the renormalization term} \label{sec:concrenom}
In this section, we will prove that $Y(x,t)$ behaves like a constant multiple of $t$, and will evaluate that constant. The first step is the following law of large numbers.
\begin{lmm} \label{lem:LLN}
Fix some $a>0$, $b>0$ and $\epsilon \in (0,1/100)$. Then there is an event $\Omega_{YLLN}$ with $\P(\Omega_{YLLN}) = 1-o(1)$, such that on $\Omega_{YLLN}$, we have that for all $(x,t)\in [-aN, aN]\times[0,bN]$,
\begin{equation*}
    |Y(x,t) - \mathbb{E}(Y(x,t))| = o(1),
\end{equation*}
where the $o(1)$ term is uniform in $(x,t)$.
\end{lmm}
\begin{proof}
Fix $(x,t)$, and write 
\begin{equation*}
Y(x,t) = \sum_{l_1,l_2} Z_{l_1,l_2},
\end{equation*}
where 
$$
Z_{l_1,l_2} = \frac{16 c}{\beta^4} \sum_{\substack{z \ \equiv \  l_1 \ (\mathrm{mod} \ \lceil 2N^{\epsilon} \rceil)\\ s\ \equiv \ l_2 \ (\mathrm{mod} \lceil N^{\epsilon} \rceil)}} p(x-z,t-s) K(z,s)^4.
$$
Let $\overline{m}(s) = \E(K(z,s)^4)$, which does not depend on $z$, as observed earlier. Define 
\begin{equation*}
\tilde{Z}_{l_1,l_2} := Z_{l_1,l_2} - \E(Z_{l_1,l_2}) = \frac{16 c}{\beta^4} \sum_{\substack{z \ \equiv \ l_1 \ (\mathrm{mod} \ \lceil 2N^{\epsilon} \rceil)\\ s\ \equiv \ l_2 \ (\mathrm{mod} \lceil N^{\epsilon} \rceil)}} p(x-z,t-s) (K(z,s)^4 - \overline{m}(s)).
\end{equation*}
Note that this is a sum of independent random variables with mean zero. Thus, as in the proof of Lemma \ref{lmm:MartBound}, we can apply the Burkholder--Davis--Gundy inequality and the Minkowski inequality to get the bound 
\begin{equation*}
(\mathbb{E}|\tilde{Z}_{l_1,l_2}|^{2p})^{1/p} \lesssim \sum_{\substack{z \ \equiv \ l_1 \ (\mathrm{mod} \ \lceil 2N^{\epsilon} \rceil)\\ s\ \equiv \ l_2 \ (\mathrm{mod} \lceil N^{\epsilon} \rceil)}} p(x-z,t-s)^2 (\mathbb{E}|K(z,s)^4 - \overline{m}(s)|^{2p})^{1/p}
\end{equation*}
From the proof of Lemma~\ref{lem:Ydiff}, recall that $\mathbb{E}|K(y,s)^4 - \overline{m}(s)|^{2p} \lesssim N^{-2p}$. Also, 
\begin{align*}
\sum_{\substack{z \ \equiv \ l_1 \ (\mathrm{mod} \ \lceil 2N^{\epsilon} \rceil)\\ s\ \equiv \ l_2 \ (\mathrm{mod} \lceil N^{\epsilon} \rceil)}} p(x-z,t-s)^2 &\le \sum_{z\in Z} \sum_{s=0}^t p(z,s)^2,
\end{align*}
which is the expected number of intersections of two simple symmetric random walks up to time $t$, when both are started at the origin. By \eqref{eq:rwintersect}, this is bounded above by a constant times $\sqrt{N}$. Combining, we get

$$
(\mathbb{E}|\tilde{Z}_{l_1,l_2}|^{2p})^{1/p} \lesssim N^{-3/2}.
$$
Thus, by Minkowski's inequality,
\begin{align*}
\|Y(x,t)-\ee(Y(x,t))\|_{L^{2p}} &\le \sum_{l_1,l_2} \|\tilde{Z}_{l_1,l_2}\|_{L^{2p}}\\
&\le N^{2\epsilon} N^{-3/4}.
\end{align*}
This allows us to choose a sufficiently high $p$, apply Markov's inequality and a union bound, and get the desired result.
\end{proof}


\begin{lmm} \label{lem:Yconc}
Fix some $a>0$, $b>0$, and $\epsilon \in (0,1/100)$. Then for any $(x,t)\in [-aN, aN]\times[0,bN]$, we have  $\E(Y(x,t)) = Vt/N + o(1)$, where $V$ is the deterministic constant defined in equation~\eqref{eq:vdef} and the $o(1)$ term is uniformly bounded in $(x,t)$ in the above region. 
\end{lmm}
\begin{proof}
First, note that since the $\xi(z,s)$ variables are independent and have mean zero, 
\begin{equation*}
\begin{aligned}
    \mathbb{E}(K(x,t)^4) & = \frac{1}{16}\sum_{z\in \Z}\sum_{t - N^{\epsilon} \le s \le t} \Delta(x-z,t-s)^4 (\mathbb{E}(\xi(z,s)^4) - (\mathbb{E}(\xi(z,s)^2))^2) \\ 
    & \qquad + \frac{1}{16}\biggl(\sum_{z\in \Z}\sum_{ t - N^{\epsilon} \le s \le t} \Delta(x-z,t-s)^2 \mathbb{E}(\xi(z,s)^2)\biggr)^2.
\end{aligned}
\end{equation*}
Since $m$ is finite in a neighborhood of zero and $\E(y(z,s))=0$, it follows that $m(\theta)=1+ O(\theta^2)$ for $\theta$ close to zero. Thus, for any positive integer $k$,
\begin{align*}
&\E(\xi(z,s)^k) = \frac{1}{m(\beta N^{-1/4})^k} \ee[(e^{\beta N^{-1/4} y(z,s)} - m(\beta N^{-1/4}))^k]\\
&= (1+O(N^{-1/2}))\E\biggl[\biggl(\beta N^{-1/4}y(z,s) + \frac{\beta^2}{2} N^{-1/2} y(z,s)^2 + \cdots - O(N^{-1/2})\biggr)^k\biggr]\\
&= (1+O(N^{-1/2}))(\E(\beta^k N^{-k/4} y(z,s)^k) + o(N^{-k/4}))\\
&= \beta^k N^{-k/4} \mu_k + o(N^{-k/4}), 
\end{align*}
where the exchange of expectation and series sum can be easily justified using the finiteness of $m$ in a neighborhood of zero and the dominated convergence theorem. In particular, $\mathbb{E}(\xi(z,s)^2) = \beta^2 \mu_2 N^{-1/2} +o(N^{-1/2})$ and $\mathbb{E}(\xi(z,s)^4) = \beta^4\mu_4 N^{-1} + o(N^{-1})$.

Now, by Lemma \ref{lmm:DeltSum}, 
\begin{align*}
\sum_{r \ge N^{\epsilon}} \sum_{z} \Delta(z,r)^2 &\lesssim  N^{-\epsilon/2},\ \ \ \sum_{r \ge 0} \sum_{z} \Delta(z,r)^2 \lesssim 1.
\end{align*}
Furthermore, using the fact that every $|\Delta(y,r)| \le 2$, we have 
\begin{align*}
    \sum_{r \ge N^{\epsilon}} \sum_z \Delta(z,r)^4 &\le 4 \sum_{r \ge N^{\epsilon}} \sum_z \Delta(z,r)^2 \lesssim N^{- \epsilon/2},\\
    \sum_{r \ge 0} \sum_z \Delta(z,r)^4 &\le 4 \sum_{r\ge 0} \sum_z \Delta(z,r)^2 \lesssim 1.
\end{align*}
Combining the above observations, we see that 
\begin{align*}
&\sum_z \sum_{t - N^{\epsilon} \le s \le t} \Delta(x-z,t-s)^4 (\mathbb{E}(\xi(z,s)^4) - (\mathbb{E}(\xi(z,s)^2))^2)\\
&\qquad \qquad - \sum_z\sum_{r\ge 0}  \Delta(z,r)^4 N^{-1} \beta^4(\mu_4 - \mu_2^2) \\ 
&= \sum_z \sum_{t - N^{\epsilon} \le s \le t} \Delta(x-z,t-s)^4[(\mathbb{E}(\xi(z,s)^4) - (\mathbb{E}(\xi(z,s)^2))^2) - N^{-1}\beta^4(\mu_4 - \mu_2^2)] \\
&\qquad \qquad - \sum_z \sum_{s > N^{\epsilon}} \Delta(z,s)^4 N^{-1}\beta^4(\mu_4 - \mu_2^2)
\\ 
&= o(N^{-1}) + O(N^{-1-\epsilon/2}) = o(N^{-1}). 
\end{align*}
Similarly,
$$
\sum_z\sum_{t- N^{\epsilon} \le s \le t} \Delta(x-z,t-s)^2 \mathbb{E}(\xi(z,s)^2) - \sum_{z} \sum_{r\ge0 }\Delta(z,r)^2 N^{-1/2} \beta^2 \mu_2 = o(N^{-1/2}).
$$
Combining all of the above, we finally have 
\begin{align*}
    \mathbb{E}(K(x,t)^4) &=\frac{\beta^4}{16N} \biggl[ \sum_{z,r} \Delta(z,r)^4 (\mu_4 - \mu_2^2)  + \biggl(\sum_{z,r} \Delta(z,r)^2 \mu_2\biggr)^2\biggr] + o(N^{-1})\\
    &= \frac{\beta^4V}{16Nc} + o(N^{-1}).
\end{align*}
To complete the proof, note that by \eqref{eq:ydef} and the above display, 
\begin{align*}
\ee(Y(x,t)) &= \frac{16c}{\beta^4} \sum_{z\in \Z}\sum_{s=1}^t p(x-z,t-s) \ee(K(z,s)^4)\\
&= V\sum_{z\in \Z}\sum_{s=1}^t p(x-z,t-s) (1+o(N^{-1}))\\
&= Vt + o(1),
\end{align*}
where the last line holds because the sum of $p(z,s)$ over all $z$ equals $1$ for any given $s$, and $t=O_n$.
\end{proof}

\section{Completing the proof of Theorem \ref{thm:mainresult}}\label{sec:mainproof}
In this section, we will complete the proof of Theorem \ref{thm:mainresult} under the assumption that $\beta \ne 0$. The $\beta=0$ case will be handled in Section \ref{sec:betazero}. Recall the function $\tf_N$ defined in equation \eqref{eq:tfndef} at points $(x,t)$ such that $x$ is an integer multiple of $N^{-1/2}$ and $t$ is an integer multiple of $N^{-1}$. As promised in the sentence below equation~\eqref{eq:tfndef0}, we now describe the method of extending the domain of $\tf_N$ to the whole of $\R \times \R_+$ by linear interpolation.

Let $g$ be any function defined on $\Z \times \Z_+$. We will now describe a way of extending the domain of $g$ to $\R \times \R_+$ by linear interpolation. It will be clear how the same prescription will apply to functions defined on $N^{-1/2}\Z \times N^{-1}\Z_+$, such as $\tf_N$. 


Firstly, we construct a graph $G$ with vertex set $V$ consisting of the set of points $(x,t) \in \mathbb{Z} \times \Z_+$ such that $x+t$ is even. (One might also consider those such that $x+t$ is odd; the important thing to notice is that the value of $\tf_N$ at the even points are completely independent of those at the odd points.) There are two types of edges in the set $E$ of edges of this graph. The first type connects $(x,t)$ to $(x-1,t-1)$ or $(x+1,t-1)$. The second type connects $(x,t)$ to $(x+2,t)$.

The graph $G$ gives a triangulation of the plane $\mathbb{R} \times \mathbb{R}_+$. Let $p$ be a point in $\mathbb{R} \times \mathbb{R}_+$. It is contained in some triangle $T$ in the triangulation given by the graph $G$. Let $a$,$b$, and $c$ be the boundary vertices of $T$. Then $p$ can be written uniquely in barycentric coordinates as a convex combination of $a$, $b$ and $c$; namely, $p= s_1 a + s_2 b +s_3 c$ with $s_1\ge0$, $s_2\ge 0$, $s_3\ge 0$ and $s_1 + s_2 + s_3 =1$. We  define $g(p):= s_1 g(a) + s_2 g(b) + s_3 g(c)$. It is easy to see that this linear interpolation is well-defined even if $p$ belongs to multiple triangles (i.e., even if $p$ is on the boundary of some triangle). 

Using the above technique, we extend the domain of $\tf_N$ to $\R \times \R_+$. Next, recall the function $\fpoly$ defined in Section \ref{sec:polymer}. Define $\tf_N^{\textup{poly}}: N^{-1/2}\Z \times N^{-1}Z\to \rr$ as 
\[
\tf_N^{\textup{poly}}(x,t) := \fpoly(\sqrt{N}x, Nt).
\]
Extend the domain of $\tf_N^{\textup{poly}}$ to $\R\times \R_+$ using the above interpolation method. Now fix any compact region $K\subseteq \R\times \R_+$. We claim that 
\begin{equation}\label{eq:prob1}
\sup_{(x,t)\in K} |\tf_N(x,t) - \tf_N^{\textup{poly}}(x,t)| \to 0
\end{equation}
in probability as $N\to\infty$. By the nature of the interpolation, it suffices to replace the supremum above by the supremum over all $(x,t)\in K_N := K\cap (N^{-1/2}\Z \times N^{-1}Z_+)$. Take any $(x,t)\in K_N$. Let $x' := \sqrt{N} x$ and $t' := Nt$. Then 
\begin{align*}
\tf_N(x,t) -  \tf_N^{\textup{poly}}(x,t) &= f_N(x', t') - \frac{Vt'}{N} - \frac{t'}{\beta}\log m(N^{-1/4} \beta) - t'  \psi(0,0)- \fpoly(x',t').
\end{align*}
By \eqref{eq:tfndef}, \eqref{eq:fxtdef} and the assumption that $\psi(0,0)=0$, this gives
\begin{align*}
\tf_N(x,t) -  \tf_N^{\textup{poly}}(x,t)  &= f(x',t') -\frac{Vt'}{N} - \fpoly(x',t'),
\end{align*}
where $f$ is the function defined in equation \eqref{eq:fxtdef}. By Theorem \ref{thm:errorprop2}, Lemma \ref{lem:LLN}, Lemma \ref{lem:Yconc}, and the above identity, we get \eqref{eq:prob1}. From \eqref{eq:prob1}, it follows that $\tf_N -  \tf_N^{\textup{poly}}\to 0$ in probability as a sequence of $C(\R\times \R_+)$-valued random variables, under the topology of uniform convergence on compact sets. Combining this with the fact that $\exp(\beta \tf_N^{\textup{poly}})$ converges in law to the solution of the stochastic heat equation \eqref{eq:she} with multiplicative noise (by \cite[Theorem 2.7]{albertsetal14b}, and the fact that having noise variables with variance $\mu_2$ at inverse temperature $\beta N^{-1/4}$ in the polymer model is equivalent to having noise variables with variance $1$ at inverse temperature $\beta \sqrt{\mu_2}$), completes the proof of Theorem \ref{thm:mainresult}. We remark here that though the statement of  \cite[Theorem 2.7]{albertsetal14b} only discusses the point-to-point partition function, their proof also holds verbatim for the point-to-line partition function, as mentioned in  \cite[Section 6.2]{albertsetal14b}. 


\section{The $\beta=0$ case}\label{sec:betazero} 
The $\beta=0$ case is much simpler than the $\beta\ne0$ case, so we will just briefly outline the modifications needed for the proof to go through. First, we need to make some changes to the definitions from Section \ref{sec:polymer}. First, we take $\poly \equiv 0$ and treat $\beta^{-1}\log m(\beta N^{-1/4})$ as zero, so that $\fpoly$ now satisfies the simple recursion 
\begin{equation*}
    \fpoly(x,t) = \frac{1}{2}(\fpoly(x-1,t-1)+ \fpoly(x+1,t-1)) + N^{-1/4}y(x,t),
\end{equation*}
and $f$ satisfies 
\begin{equation*}
    f(x,t) =  \psi(f(x-1,t-1), f(x+1,t-1)) + N^{-1/4}y(x,t).
\end{equation*}
The explicit expression for $\fpoly(x,t)$ is now
\begin{equation*}
    \fpoly(x,t) = \sum_{z \in \mathbb{Z}} \sum_{s=1}^t p(z-x,t-s) N^{-1/4} y(x,t),
\end{equation*}
which is just a linear combination of i.i.d.~random variables and therefore much easier to analyze. 
If we now define $K(x,t)$ as
\begin{equation*}
    K(x,t):= \sum_{z \in \mathbb{Z}} \sum_{t -N^{\epsilon}\le s\le t} \Delta(z-x,t-s) N^{-1/4} y(x,t)
\end{equation*}
and $Y(x,t)$ as
\begin{equation*}
    Y(x,t):= c \sum_{z \in \mathbb{Z}} \sum_{s=1}^t  p(z-x, t-s) K(z,s)^4,
\end{equation*}
then using similar (but simpler) arguments as before, we can show that 
\begin{equation*}
    f(x,t) = \fpoly(x,t) + Y(x,t) + \delta(x,t),
\end{equation*}
where $\delta(x,t) = o(1)$ uniformly in $(x,t)\in [-aN, aN]\times[0,bN]$ with high probability. (Note that the new $K$ is not quite the limit of the $K$ from earlier as $\beta \to0$. It is the limit of the old $K$ divided by $\beta$ as $\beta \to 0$.)

As before, we can follow the steps of Theorem \ref{thm:errorprop2} to get an inductive relationship on $\delta$. There are two major error propagation terms that do not get incorporated into a multiplicative factor of~$\delta$.  The first comes from the difference $\fpoly_1 - \fpoly_2$. 
From direct computation, we have 
\begin{equation*}
    \fpoly_1 - \fpoly_2 = \sum_{z \in \mathbb{Z}} \sum_{s=1}^t \Delta(z-x, t-s) N^{-1/4} y(x,t).
\end{equation*}
This is a linear combination of i.i.d.~random variables, and so one can use standard moment bounds (much like in the proofs of Theorem \ref{thm:upperbnd} and Lemma \ref{lmm:MartBound}) to assert that the above quantity should be $O(N^{-1/4 + \epsilon})$ uniformly in a large rectangle $[-aN,aN] \times [0,bN]$. 

The second error term that one has to deal with is the term 
\begin{equation*}
    (f_1 - f_2)^4 - K(x,t)^4.
\end{equation*}
This is $O(N^{-1-\epsilon})$ via the methods in the proof of Corollary \ref{cor:Xd4}. With these two estimates in hand, one can follow the proof of \ref{thm:errorprop2} verbatim to establish the required inductive estimate. 

Finally, one has to show that $Y(x,t) = Vt/N + o(1)$. This is done by the arguments of Section \ref{sec:concrenom}, which go through without any trouble. 

%
%

\begin{acks}[Acknowledgments]
A.A.~thanks Izumi Okada for useful discussions. S.C.~thanks Kevin Yang, Hao Shen, Persi Diaconis  and Peter Friz for helpful comments and references. Both of us thank the two anonymous referees for numerous helpful remarks.
\end{acks}
\begin{funding}
 The first author was supported by NSF grant 2102842.
 The second author was supported in part by NSF grants 1855484, 2113242, and 2153654.
\end{funding}

\bibliographystyle{imsart-number}
\bibliography{myrefs-KPZ}


\end{document}